\documentclass[reqno]{amsart}
\usepackage{graphicx}

\usepackage{amssymb}
\usepackage{amsfonts}
\usepackage{indentfirst}
\usepackage{graphicx, graphics, psfig, PSTricks}
\usepackage{pst-plot}

\renewcommand\bigskip{\medskip}
\def\@oddhead{\hbox to \textwidth{\footnotesize{\it
Khintchine exponents and Lyapunov exponents } \hfill\thepage}}

\def\bc{\begin{center}}
\def\ec{\end{center}}

\newtheorem{lem}{\bf Lemma}[section]

\newtheorem{pro}[lem]{\bf Proposition}
\newtheorem{thm}[lem]{\bf Theorem}
\newtheorem{rem}[lem]{\bf Remark}

\newtheorem{cor}[lem]{\bf Corollary}
\newtheorem{fac}[lem]{\bf Fact}

\begin{document}

\title[Khintchine exponents and Lyapunov
exponents]{On Khintchine exponents and Lyapunov exponents of
continued  fractions}

\author{Ai-Hua FAN}
\address{Ai-Hua FAN: \
Department of Mathematics,
 Wuhan University,
 Wuhan, 430072, P.R. China
 \& CNRS UMR 6140-LAMFA,
 Universit\'e de Picardie
 80039 Amiens, France}%
\email{ai-hua.fan@u-picaride.fr}%
\author{Ling-Min LIAO}
\address{Ling-Min LIAO: Department of Mathematics,
 Wuhan University,
 Wuhan, 430072, P.R. China \&
 CNRS UMR 6140-LAMFA,
 Universit\'e de Picardie
 80039 Amiens, France}%
 \email{lingmin.liao@u-picardie.fr}%
\author{Bao-Wei WANG}
\address{Bao-Wei WANG: Department of Mathematics,
 Huazhong University of Science and Technology,
 Wuhan, 430074, P.R. China}%
 \email{bwei$_-$wang@yahoo.com.cn}%
\author{Jun WU}
\address{Jun WU: Department of Mathematics,
 Huazhong University of Science and Technology,
 Wuhan, 430074, P.R. China}%
 \email{wujunyu@public.wh.hb.cn }%

\subjclass[2000]{11K55, 28A78, 28A80}
 \keywords{Continued fraction, Gibbs measure, Hausdorff
dimension}
\begin{abstract}
Assume that $x\in [0,1) $ admits its continued fraction expansion
$x=[a_1(x), a_2(x),\cdots]$.  The Khintchine exponent $\gamma(x)$ of
$x$ is defined by $ \gamma(x):=\lim\limits_{n\to \infty}\frac{1}{n}
\sum_{j=1}^n \log a_j(x)$ when the limit exists. Khintchine spectrum
$\dim E_\xi$ is fully studied, where $ E_{\xi}:=\{x\in
[0,1):\gamma(x)=\xi\} \ (\xi \geq 0)$ and $\dim$ denotes the
Hausdorff dimension. In particular, we prove the remarkable fact
that the Khintchine spectrum $\dim E_{\xi}$, as function of $\xi \in
[0, +\infty)$, is neither concave nor convex. This is a new
phenomenon from the usual point of view of multifractal analysis.
Fast Khintchine exponents defined by
$\gamma^{\varphi}(x):=\lim\limits_{n\to \infty}\frac{1}{\varphi(n)}
\sum_{j=1}^n \log a_j(x)$ are also studied, where $\varphi (n)$
tends to the infinity faster than $n$ does. Under some regular
conditions on $\varphi$, it is proved that the fast Khintchine
spectrum $\dim (\{ x\in [0,1]: \gamma^{\varphi}(x)= \xi \}) $ is a
constant function. Our method also works for other spectra like the
Lyapunov spectrum and the fast Lyapunov spectrum.
\end{abstract}

\maketitle

\addtocounter{section}{0}
\section{Introduction and Statements}
The continued fraction of a real number can be generated by the
Gauss transformation $T:[0,1)\rightarrow [0,1)$ defined by
\begin{eqnarray}\label {ff1}
T(0):=0,\ \ T(x):=\frac{1}{x} \ ({\rm{mod}}\ 1),\  {\rm{for}}\  x
\in (0,1)
\end{eqnarray}
in the sense that every irrational number $x$ in $[0,1)$ can be
uniquely expanded as an infinite expansion of the form
\begin{equation}\label{ff2}
x=\frac{\displaystyle 1}{\displaystyle a_1(x)+\frac{1}{a_2+\ddots+
\frac{\displaystyle 1}{\displaystyle
a_n(x)+T^n(x)}}}=\frac{\displaystyle 1}{\displaystyle a_1(x)+
\frac{\displaystyle 1}{ a_2(x)+\displaystyle \frac{ 1}{
a_3(x)+\ddots}}}
\end{equation}
where $a_1(x)=\lfloor{1}/{x}\rfloor$ and $a_n(x)=a_1(T^{n-1}(x))$
for $n\geq 2$ are called {\em partial quotients} of $x$ ($\lfloor x
\rfloor$ denoting the integral part of $x$). For simplicity, we will
denote the second term in (\ref{ff2}) by $[a_1,a_2,\cdots,
a_n+T^n(x)]$ and the third term by $[a_1,a_2, a_3,\cdots]$.

\indent It was known to E. Borel \cite{Bo1} (1909) that for Lebesgue
almost all $x\in [0,1)$, there exists a subsequence $\{a_{n_r}(x)\}$
of $\{a_n(x)\}$ such that $a_{n_r}(x)\rightarrow \infty$. A more
explicit result due to Borel-Bernstein (see \cite{BB,Bo1,Bo2}) is
the $0$-$1$ law which hints that for almost all $x\in [0,1]$,
$a_n(x)>\varphi(n)$ holds for infinitely many $n$'s or finitely many
$n$'s according to  $\sum\limits_{n\geq 1}\frac{1}{\varphi(n)}$
diverges or converges.  Then it arose a natural question to quantify
the exceptional sets in terms of Hausdorff dimension (denoted by
$\dim$). The first published work on this aspect was due to I.
Jarnik \cite{Ja} (1928) who was concerned with the set $E$ of
continued fractions with bounded partial quotients and with the sets
$E_2, E_3, \cdots$, where $E_\alpha$ is the set of continued
fractions whose partial quotients do not exceed $\alpha$. He
successfully got that the set $E$ is of full Hausdorff dimension,
but he didn't find the exact dimensions of $E_2, E_3, \cdots$.
Later, many works are done to estimate $\dim E_2$, including those
of I. J. Good \cite{Go}, R. Bumby \cite{Bu2}, D. Hensley \cite{He1,
He2}, O. Jenkinson and M. Pollicott \cite{JenP}, R. D. Mauldin, M.
Urba\'nski \cite{MU} and references therein. Up to now, the optimal
approximation on $\dim E_2$ is the result given by O. Jenkinson
\cite{Jen} (2004):
\begin{eqnarray*}
\dim
E_2=0.531280506277205141624468647368471785493059109018398779\cdots
\end{eqnarray*}
which is claimed to be accurate to 54 decimal places.

In the present paper, we study the Khintchine exponents and the
Lyapunov exponents of continued fractions. For any $x\in [0,1)$ with
its continued fraction (\ref{ff2}), we define its {\it Khintchine
exponent} $\gamma(x)$ and {\it Lyapunov exponent}
 $\lambda(x)$ respectively by
\begin{eqnarray*}
&&\gamma (x):=\lim_{n\to \infty}\frac{1}{n}\sum_{j=1}^n\log a_j(x)
=\lim_{n\to \infty}\frac{1}{n}\sum_{j=0}^{n-1}\log
a_1({T^j(x)}),\\
&&\lambda (x):=\lim_{n\to \infty}\frac{1}{n}\log
\Big|(T^n)'(x)\Big|=\lim_{n\to
\infty}\frac{1}{n}\sum_{j=0}^{n-1}\log \Big|T'(T^j(x))\Big|,
\end{eqnarray*}
if the limits exist. The Khintchine exponent of $x$ stands for the
average (geometric) growth rate of the partial quotients $a_n(x)$,
and the Lyapunov exponent which is extensively studied from
dynamical system point of view, stands for the expanding rate of
$T$. Their common feature is  that both are Birkhoff averages.

Let $\varphi: \mathbb{N} \to \mathbb{R}_+$. Assume that $\lim_{n\to
\infty} \frac{\varphi(n)}{n}= \infty$.  The {\it fast Khintchine
exponent} and {\it fast Lyapunov exponent} of $x\in [0,1]$, relative
to $\varphi$, are respectively defined by
\begin{eqnarray*}
&&\gamma^{\varphi} (x):=\lim_{n\to
\infty}\frac{1}{\varphi(n)}\sum_{j=1}^n\log a_j(x) =\lim_{n\to
\infty}\frac{1}{\varphi(n)}\sum_{j=0}^{n-1}\log
a_1({T^j(x)}),\\
&&\lambda^{\varphi} (x):=\lim_{n\to \infty}\frac{1}{\varphi(n)}\log
\Big|(T^n)'(x)\Big|=\lim_{n\to
\infty}\frac{1}{\varphi(n)}\sum_{j=0}^{n-1}\log
\Big|T'(T^j(x))\Big|.
\end{eqnarray*}

It is well known (see \cite{Bi, Wa}) that the transformation $T$ is
measure preserving and ergodic with respect to the Gauss measure
$\mu_G$ defined as
\begin{eqnarray*}
d\mu_G=\frac{dx}{(1+x)\log 2}.
\end{eqnarray*}
 An application of Birkhoff ergodic theorem yields
that for Lebesgue almost all $x\in [0,1)$,
\begin{eqnarray*}
 &&\gamma(x)=\xi_0=\int\log a_1(x)d\mu_G =\frac{1}{\log 2} \sum_{n=1}^\infty \log n \cdot \log \Big( 1+ \frac{1}{n(n+2)} \Big) =2.6854...\\
&&\lambda (x)=\lambda_0=\int\log |T'(x)|d\mu_G = \frac{\pi^2}{6
\log 2} = 2.37314....
\end{eqnarray*}
Here $\xi_0$ is called the {\it{Khintchine  constant}} and
$\lambda_0$ the {\it{Lyapunov  constant}}. Both constants are
relative to the Gauss measure.

For real numbers $\xi,\beta \geq 0$, we are interested in the level
sets of Khintchine exponents and Lyapunov exponents:
\begin{eqnarray*}
&& E_{\xi}:=\{x\in [0,1): \gamma(x)=\xi\},\\
&& F_{\beta}:=\{x\in [0,1): \lambda(x)=\beta\}.
\end{eqnarray*}
We are also interested in the level sets of fast Khintchine
exponents and fast Lyapunov exponents:
\begin{eqnarray*}
&&E_{\xi}(\varphi):=\{x\in [0,1): \gamma^{\varphi} (x)=\xi\},\\
&&F_\beta(\varphi):=\{x\in [0,1): \lambda^{\varphi} (x)=\beta\}.
\end{eqnarray*}

The {\it Khintchine spectrum} and the {\it Lyapunov spectrum} are
the dimensional functions:
\begin{eqnarray*}
t(\xi):= \dim E_{\xi} \qquad \tilde{t}(\xi):= \dim F_{\xi}.
\end{eqnarray*}
The following two functions
\begin{eqnarray*}
t^{\varphi}(\xi):= \dim E_{\xi}(\varphi) \qquad
\tilde{t}^{\varphi}(\xi):= \dim F_{\xi}(\varphi)
\end{eqnarray*}
are called the {\it fast Khintchine spectrum} and the {\it fast
Lyapunov spectrum} relative to $\varphi$.

 M. Pollicott and H. Weiss \cite{PoW} initially studied the
level set of $F_{\beta}$ and obtained some partial results about the
function $t(\xi)$. In the present work, we will give a complete
study on the Khintchine spectrum and the Lyapunov spectrum. Fast
Khintchine spectrum  and fast Lyapunov spectrum are considered here
for the first time. We shall see that both functions
$t^{\varphi}(\xi)$ and $\tilde{t}^{\varphi}(\xi)$ are equal.

We start with the statement of our results on fast spectra.

\begin{thm}\label{thm-fast}
 Suppose $\left(\varphi(n+1)-\varphi(n)\right)
\uparrow \infty$ and $\lim\limits_{n\to
\infty}\frac{\varphi(n+1)}{\varphi(n)}:=b \geq 1 $. Then
$E_{\xi}(\varphi)=F_{2\xi}(\varphi)$ and $\dim E_{\xi}(\varphi) =
1/(b+1) $ for all $\xi\geq 0$.
\end{thm}

In order to state our results on the Khintchine spectrum, let us
first introduce some notation. Let
$$D:=\{(t,q) \in \mathbb{R}^2: 2t-q
> 1 \}, \qquad D_0:=
\{(t,q)\in \mathbb{R}^2: 2t-q>1, 0\leq t \leq 1 \}.$$
 For $(t,q)\in D$, define
\begin{eqnarray*}
P(t,q):=\lim_{n\to \infty} \frac{1}{n} \log
\sum_{\omega_1=1}^{\infty}\cdots\sum_{\omega_n=1}^{\infty} \exp
\left(
  \sup_{x\in [0,1]} \log \prod_{j=1}^{n}
\omega_j^q([\omega_j,\cdots, \omega_n+x])^{2t}\right).
\end{eqnarray*}
 It will be proved that $P(t,q)$ is an analytic function in $D$
 (Proposition \ref{measure_potential}).

 Moreover, for any $\xi\geq 0$, there exists a unique solution
$(t(\xi),q(\xi))\in D_0$ to the equation
\begin{eqnarray*}\left\{
  \begin{array}{ll}
     P(t,q)=q\xi, \\
     \displaystyle \frac{\partial P}{\partial q}(t,q)= \xi.
  \end{array}
\right.
\end{eqnarray*}
(Proposition~\ref{solution}).

\begin{thm}\label{thm-Birkhoff}
Let $\xi_0 = \int \log a_1(x) d \mu_G(x)$.
 For $\xi \geq 0 $, the set $E_{\xi}$ is of Hausdorff dimension
 $t(\xi)$. Furthermore, the dimension function $t(\xi)$ has the
 following properties:\\
\indent {\rm 1)} $t(\xi_0)=1 $ and $t(+\infty)=1/2 $;\\
\indent {\rm 2)} $t'(\xi)<0$ for all $\xi>\xi_0 $, $t'(\xi_0)=0$, and  $t'(\xi)>0$ for all $\xi<\xi_0 $; \\
\indent {\rm 3)} $t'(0+)=+\infty $ and $t'(+\infty)=0 $;\\
\indent {\rm 4)} $t''(\xi_0)<0$, but $t''(\xi_1)>0 $ for some
$\xi_1>\xi_0 $, so $t(\xi)$ is neither convex nor concave.
\end{thm}

See Figure 1 for the graph of  $t(\xi)$.

\begin{center}
\begin{pspicture}(0,0)(10,7.4)
\rput(5.0,3.4){\includegraphics{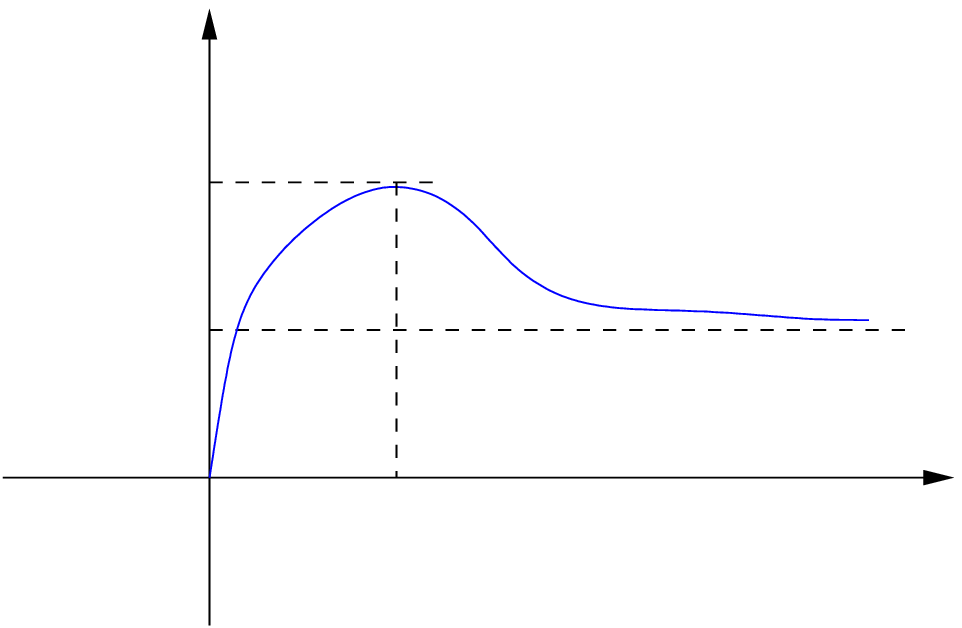}} \rput(2.05,1.5){$0$}
\rput(2.05,3.22){$\frac{1}{2}$} \rput(1.8, 6){$t(\xi)$}
\rput(2.05,4.75){$1$} \rput(4.17,1.5){$\xi_0$} \rput(9.6,1.5){$\xi$}
\end{pspicture}
\begin{center}
{{\it Figure 1.} Khintchine spectrum}
\end{center}
\end{center}

It should be noticed that the above fourth property of $t(\xi)$,
i.e. the non-convexity, shows a new phenomenon for the multifractal
analysis in our settings.

 Let $$\tilde{D}:= \{(\tilde{t},q): \tilde{t}-q >1/2 \} \qquad
 \tilde{D}_0:= \{(\tilde{t},q): \tilde{t}-q>1/2, 0\leq \tilde{t} \leq 1 \}.
 $$
 For $(\tilde{t},q)\in
\tilde{D}$, define
\begin{eqnarray*}
P_1(\tilde{t},q):=\lim_{n\to \infty} \frac{1}{n} \log
\sum_{\omega_1=1}^{\infty}\cdots\sum_{\omega_n=1}^{\infty} \exp
\Bigg(
  \sup_{x\in [0,1]} \log \prod_{j=1}^{n}
([\omega_j,\cdots, \omega_n+x])^{2(\tilde{t}-q)}\Bigg).
\end{eqnarray*}
In fact, $P_1(\tilde{t},q)=P(\tilde{t}-q,0)$, thus
$P_1(\tilde{t},q)$ is analytic in $\tilde{D}$.

Denote $\gamma_0:= 2\log \frac{1+\sqrt{5}}{2} $. For any $\beta\in
(\gamma_0, \infty)$, the system
\begin{eqnarray*}
\left\{
  \begin{array}{ll}
     P_1(\tilde{t},q)=q\beta, \\
     \displaystyle \frac{\partial P_1}{\partial q}(\tilde{t},q)= \beta
  \end{array}
\right.
\end{eqnarray*}
admits a unique solution $(\tilde{t}(\beta),q(\beta))\in
\tilde{D}_0$ (Proposition \ref{solution-Lya}).

\begin{thm}\label{thm-Lyapunov}
 Let $\lambda_0= \int \log |T'(x)| d \mu_G $ and
  $\gamma_0=2\log \frac{1+\sqrt{5}}{2}$. For any $\beta \in [\gamma_0, \infty)$,
  the set $F_{\beta}$ is of Hausdorff dimension
 $\tilde{t}(\beta)$. Furthermore the dimension function $\tilde{t}(\xi)$ has the
 following properties:\\
\indent {\rm 1)} $\tilde{t}(\lambda_0)=1 $ and $\tilde{t}(+\infty)=1/2 $;\\
\indent {\rm 2)} $\tilde{t}'(\beta)<0$ for all $\beta>\lambda_0 $,
$\tilde{t}'(\lambda_0)=0$,
 and  $\tilde{t}'(\beta)>0$ for all $\beta<\lambda_0 $; \\
\indent {\rm 3)} $\tilde{t}'(\gamma_0+)=+\infty $ and $\tilde{t}'(+\infty)=0 $;\\
\indent {\rm 4)} $\tilde{t}''(\lambda_0)<0$, but
$\tilde{t}''(\beta_1)>0 $ for some $\beta_1>\lambda_0 $, i.e.,
$\tilde{t}(\beta)$ is neither convex nor concave.
\end{thm}

See Figure 2 for the graph of $\tilde{t}(\beta)$.

\begin{center}
\begin{pspicture}(0,0)(10,7)
\rput(5.0,3.1){\includegraphics{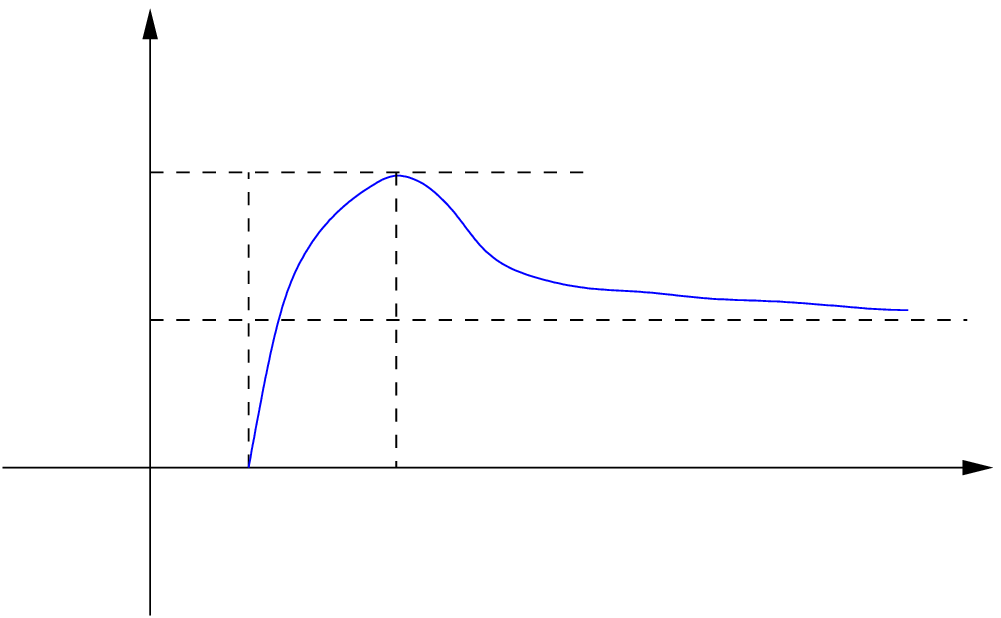}} \rput(1.25,1.27){$0$}
\rput(1.25,2.99){$\frac{1}{2}$} \rput(1, 5.77){$\tilde{t}(\beta)$}
\rput(1.25,4.52){$1$} \rput(2.44,1.2){$\gamma_0$}
\rput(4.05,1.2){$\lambda_0$} \rput(9.65,1.2){$\beta$}
\end{pspicture}
\begin{center}
{{\it Figure 2.} Lyapunov spectrum}
\end{center}
\end{center}

The last two theorems are concerned with special Birkhoff spectra.
In general,  let $(X,T)$ be a dynamical system ($T$ being a map from
a metric space $X$ into itself). The Birkhoff average of a function
$\phi: X \rightarrow \mathbb{R}$, defined by
\begin{eqnarray*}
  \overline{\phi}(x) :=\lim_{n\to \infty}\frac{1}{n} \sum_{j=0}^{n-1} \phi(T^j(x))
  \qquad x\in X
\end{eqnarray*}
(if the limit exists) is widely studied. From the point of view of
multifractal analysis, one is often interested in determining the
Hausdorff dimension of the set $\{ x\in X: \overline{\phi}(x)=
\alpha \}$ for a given $\alpha \in \mathbb{R}$. The function
\begin{eqnarray*}
f(\alpha):= \dim \left( \{ x\in X: \overline{\phi}(x)= \alpha
\}\right)
\end{eqnarray*}
 is called the {\it{Birkhoff spectrum}} for the function
$\phi$. When $X$ is compact, $T$ and $\phi$ are continuous, the
Birkhoff spectrum are well studied (see \cite{BSS,FF,FFW} and the
references therein. See also the book of Y. B. Pesin~\cite{Pesin}).

The main tool of our study is the Ruelle-Perron-Frobenius operator
with potential function
$$
\Phi_{t,q}(x)=t\log |T'(x)|+q\log a_1(x), \ \ \Psi_t(x)=t\log
|T'(x)|,
$$
where $(t, q)$ are suitable parameters.  The classical way to obtain
the spectrum through Ruelle theory usually fixes $q$ and finds
$T(q)$ as the solution of $P(T(q),q)=0$. (Here $P(t,q)$ is the
pressure corresponding to the potential function of two parameters.)
By focusing on the curve $T(q)$, one can only get some partial
results (\cite{PoW}). In the present paper, we look for multifractal
information from the whole two dimensional surface defined by the
pressure $P(t,q)$ rather than the single curve $T(q)$. This  leads
us to obtain  complete graphs of the Khintchine spectrum and
Lyapunov spectrum.

For the Gauss dynamics, there exist several works on pressure
functions associated to different potentials.  For a detailed study
on pressure function associated to one potential function, we refer
to the works of D. Mayer \cite{Ma1, Ma2, Ma3}, and for pressure
functions associated to two potential functions, we refer to the
works of M. Pollicott and H. Weiss \cite{PoW}, of P. Walters
\cite{Wa1, Wa2} and of P. Hanus, R. D. Mauldin and M. Urbanski
\cite{HMU}. We will use the theory developed in \cite{HMU}.

The paper is organized as follows. In Section 2, we collect and
establish some basic results that will be used later. Section 3 is
devoted to proving the results about the fast Khintchine spectrum
and fast Lyapunov spectrum (Theorem \ref{thm-fast}). In Section 4,
we present a general Ruelle operator theory developed in \cite{HMU}
and then apply it to the Gauss transformation. Based on Section 4,
we establish Theorem \ref{thm-Birkhoff} in Section 5. The last
section is devoted to the study of Lyapunov spectrum (Theorem
\ref{thm-Lyapunov}).

The present paper is a part of the second author's Ph. D. thesis.


\setcounter{equation}{0}

\section{Preliminary}
In this section, we collect some known facts and establish some
elementary properties of continued fractions that will be used
later. For a wealth of classical results about continued fractions,
see the books by J. Cassels \cite{Ca}, G. Hardy and E. Wright
\cite{HaW}. The books by P. Billingsley \cite{Bi}, I. Cornfeld, S.
Fomin and Ya. Sinai \cite{Co} contain an excellent introduction to
the dynamics of the Gauss transformations and its connection with
Diophantine approximation.

\subsection{Elementary properties of continued fractions}
Denote by $p_n/q_n$ the usual $n$-th $convergent$ of continued
fraction $x=[a_1(x), a_2(x), \cdots]\in [0,1)\setminus \mathbb{Q}$,
defined by
\begin{eqnarray*}
 \frac{p_n}{q_n}:=[ a_1(x), \cdots, a_n(x)]
 :=\frac{\displaystyle 1}{\displaystyle a_1(x)+
 \frac{\displaystyle 1}{\displaystyle a_2(x)+\ddots+\displaystyle
 \frac{1}{a_n(x)}}}.
 \end{eqnarray*}
It is known (see \cite{Kh} p.9) that $p_n, q_n$ can be obtained by
the recursive relation:
\begin{eqnarray*}
p_{-1}=1, \ p_0=0, \ p_n = a_np_{n-1}+p_{n-2}  & & (n\geq 2),\\
q_{-1}=0, \ q_0=1, \ q_n = a_nq_{n-1}+q_{n-2}  & & (n\geq 2).
\end{eqnarray*}
Furthermore, we have
\begin{lem}[\cite{Kr1} p.5]\label{Q-equ}
 Let $\varepsilon_1, \cdots, \varepsilon_n \in \mathbb{R}^+$.
Define inductively
\begin{eqnarray*}
Q_{-1}=0, \ Q_0=1,  \ Q_n(\varepsilon_1, \cdots,
\varepsilon_n)=\varepsilon_n Q_{n-1}(\varepsilon_1, \cdots,
\varepsilon_{n-1})+Q_{n-2}(\varepsilon_1, \cdots,
\varepsilon_{n-2}).
\end{eqnarray*}
($Q_n$ is commonly called a continuant.)
Then we have \\
\indent {\rm (i)}   $Q_n(\varepsilon_1, \cdots, \varepsilon_n)=Q_n(\varepsilon_n, \cdots, \varepsilon_1)$;\\
\indent {\rm (ii)}  $q_n=Q_n(a_1,\cdots, a_n)$,
$p_n=Q_{n-1}(a_2,\cdots, a_n)$.
\end{lem}

As consequences, we have the following results.

\begin{lem}[\cite{Kh}]\label{l2.2} For any
$a_1,a_2,\cdots,a_n, b_1,\cdots, b_{m}\in \mathbb{N}$, let
$q_n=q_n(a_1,\cdots,a_n)$ and $p_n=p_n(a_1,\cdots,a_n)$. We have

 \mbox{\rm (i)} \ $p_{n-1}q_n-p_nq_{n-1}=(-1)^n$;

\mbox{\rm (ii)} \ $q_{n+m}(a_1,\cdots,a_n, b_1,\cdots,
b_{m})=q_n(a_1,\cdots,
a_n)q_{m}(b_1,\cdots,b_{m})+\\
\indent\ \ \ \ \ \ \ \ \
\qquad\qquad\qquad\qquad\qquad\qquad\qquad\qquad
q_{n-1}(a_1,\cdots,a_{n-1})p_{m-1}(b_1,\cdots,b_{m-1})$;

\mbox{\rm (iii)} \ $q_n\geq 2^{\frac{n-1}{2}}$, \ \ \
$\prod\limits_{k=1}^na_k\leq q_n\leq \prod\limits_{k=1}^n(a_k+1).$
\end{lem}
\begin{lem}[\cite{Wu}]\label{l2.3}For any $a_1, a_2,\cdots, a_n, b\in \mathbb{N}$,
\begin{eqnarray*} \frac{b+1}{2}\leq \frac{q_{n+1}(a_1,\cdots, a_j, b,
a_{j+1}, \cdots, a_n)}{q_n(a_1,\cdots, a_j, a_{j+1},\cdots,
a_n)}\leq b+1 \qquad (\forall 1\leq j <n).
\end{eqnarray*}\end{lem}

For any $a_1, a_2,\cdots, a_n\in \mathbb{N}$, let
\begin{equation} I_n(a_1, a_2, \cdots, a_n)=\{x\in [0,1): a_1(x)=a_1,
a_2(x)=a_2, \cdots, a_n(x)=a_n\}
\end{equation}
which is called an {\em $n$-th order cylinder}.

\begin{lem}[\cite{Kr1} p.18] \label{length}
For any $a_1, a_2,\cdots, a_n\in \mathbb{N}$, the $n$-th order
cylinder \\ $I_n(a_1, a_2, \cdots, a_n)$ is the interval with the
endpoints $p_n/q_n$ and $(p_n+p_{n-1})/(q_n+q_{n-1})$. As a
consequence, the length of $I_n(a_1,\cdots, a_n)$ is equal to
\begin{eqnarray}\label{length-interval}
\Big|I_n(a_1,\cdots, a_n)\Big|=\frac{1}{q_n(q_n+q_{n-1})}.
\end{eqnarray}
\end{lem}

We will denote $I_n(x)$ the $n$-th order cylinder that contains $x$,
i.e. $I_n(x)=I_n\big(a_1(x),\cdots,a_n(x)\big)$. Let $B(x,r)$
denotes the ball centered at $x$ with radius $r$. For any $x\in
I_n(a_1,\cdots,a_n)$, we have the following relationship between the
ball $B(x,|I_n(a_1,\cdots,a_n)|)$ and $I_n(a_1,\cdots,a_n)$, which
is called the {\em {regular property}} in \cite{BDK}.
\begin{lem}[\cite{BDK}] \label{l2.5}
Let $x=[a_1, a_2,\cdots]$. We have:\\
\indent \mbox{\rm (i)} \ if $a_n\neq 1$, $B(x,|I_n(x)|)\subset
\bigcup\limits_{j=-1}^3I_n(a_1,\cdots,a_n+j)$;\\
 \indent \mbox{\rm
(ii)} \ if $a_n=1$ and $a_{n-1}\neq 1$, $B(x,|I_n(x)|)\subset
\bigcup\limits_{j=-1}^3I_{n-1}(a_1,\cdots,a_{n-1}+j)$;\\
 \indent
\mbox{\rm (iii)} if $a_n=1$ and $a_{n-1}=1$, $B(x,|I_n(x)|)\subset
I_{n-2}(a_1,\cdots,a_{n-2})$.
\end{lem}

 The Gauss transformation $T$ admits the following $\it
Jacobian \ estimate $.
\begin{lem}\label{p2}
There exists a positive number $K>1$ such that for all irrational
number $x$ in $[0,1)$, one has
\begin{eqnarray*}
0<\frac{1}{K}\leq \sup_{n\geq 0}\sup_{y\in
I_n(x)}\left|\frac{(T^n)'(x)}{(T^n)'(y)}\right|\leq K<\infty.
\end{eqnarray*}
\end{lem}
\begin{proof} \ Assume $x=[a_1, \cdots, a_n, \cdots]\in [0,1) \setminus
\mathbb{Q}$. For any $n\geq 0$ and $y\in I_n(x)=I_n(a_1,\cdots,
a_n)$, by the fact that $T'(x)=- \frac{1}{x^2} $ we get
\begin{eqnarray*}
\sum_{j=0}^{n-1}\Big|\log \big|T'(T^j(x))\big|-\log
\big|T'(T^j(y))\big|\Big|=2\sum_{j=0}^{n-1}\Big|\log T^j(x)-\log
T^j(y)\Big|.
\end{eqnarray*}
Applying the mean-value theorem, we have
$$\big|\log T^j(x)-\log
T^j(y)\big|=\left|\frac{T^j(x)-T^j(y)}{T^j(z)}\right|\leq
\frac{a_{j+1}}{q_{n-j}(a_{j+1},\cdots, a_n)},$$
 where the assertion follows from the fact that all three points
 $T^j(x), T^j(y)$ and $T^j(z)$ belong to $I_{n-j}(a_{j+1}, \cdots,
 a_n)$. By Lemma \ref{l2.2}, we have
\begin{eqnarray*} \sum_{j=0}^{n-1}\left|\log T^j(x)-\log
T^j(y)\right|\leq\sum_{j=0}^{n-1}\frac{1}{q_{n-j-1}(a_{j+2},\cdots,
a_n)}\leq \sum_{j=0}^{n-1}\left(\frac{1}{2}\right)^{n-j-2}\leq 4.
\end{eqnarray*}
Thus the  result is proved with $K=e^4$. \end{proof}

The above Jacobian estimate property of $T$ enables us to control
the length of $I_n(x)$ by $|(T^n)'(x)|^{-1}$, through the fact that
$ \int_{I_n(x)} |(T^n)'(y)| dy =1$.
\begin{lem}\label{lemma2.5}
There exist a positive constant $K>0$ such that for all irrational
numbers $x$ in $[0,1)$,
\begin{eqnarray*} \frac{1}{K}\leq
\frac{|I_n(x)|}{|(T^n)'(x)|^{-1}}\leq K.
\end{eqnarray*}
\end{lem}

We remark that from Lemma \ref{length} and Lemma \ref{lemma2.5}, we
have
\begin{eqnarray*}
\frac{1}{2K}q_n^2(x)\leq\Big|(T^n)'(x)\Big|\leq Kq_n^2(x).
\end{eqnarray*}
So the Lyapunov exponent $\lambda(x)$ is nothing but the growth rate
of $q_n(x)$ up to a multiplicative constant $2$:
\begin{eqnarray*}
\lambda(x)=\lim_{n\to \infty} \frac{2}{n} \log q_n(x).
\end{eqnarray*}

For any irrational number $x$ in $[0,1) $, let $N_n(x):= \{j\leq n:
a_j(x)\neq 1 \} $. Set
\begin{eqnarray*}
& &A:=\Big\{ x\in [0,1]: \lim_{n\to \infty} \frac{1}{n}\log
q_n(x)=\frac{\gamma_0}{2} \Big\},\\
& &B:=\Big\{ x\in [0,1]: \lim_{n\to \infty}
\frac{1}{n}\sum_{j=1}^{n}
\log a_j(x) =0 \Big\},\\
& &C:=\Big\{ x\in [0,1]: \lim_{n\to \infty} \frac{1}{n}\sharp N_n(x)
=0 \Big\},
\end{eqnarray*}
where $\sharp$ stands for the cardinal of a set. Then we have the
following relationship.
\begin{lem}\label{ABC}
With the notations given above, we have
\begin{eqnarray*}
A=B \subset C.
\end{eqnarray*}
\end{lem}
\begin{proof} It is clear that $A \subset C $ and $B\subset C $. Let
us prove $A=B $. First observe that, by Lemma \ref{l2.3}, we have
\begin{eqnarray*}
 \frac{1}{n} \log q_n(x) &\geq& \frac{1}{n} \sum_{j\in N_n(x)} \log \frac{a_j(x)+1}{2} + \frac{1}{n} \log q_{n-\sharp
 N_n}(1,\dots,1)\\
 &\geq&\frac{1}{n} \sum_{j\in N_n(x)} \log a_j(x) - \frac{1}{n}\sum_{j\in N_n(x)} \log 2 + \frac{1}{n} \log q_{n-\sharp
 N_n}(1,\dots,1).
\end{eqnarray*}
Assume $x\in A$. Since $A\subset C $, we have
\begin{eqnarray*}
 - \frac{1}{n}\sum_{j\in N_n(x)} \log 2 +
\frac{1}{n} \log q_{n-\sharp
 N_n}(1,\dots,1) \longrightarrow  0+\frac{\gamma_0}{2} \qquad (n\to \infty).
\end{eqnarray*}
Now by the assumption  $x\in A$, it follows
\begin{eqnarray*}
 \lim_{n\to \infty}
\frac{1}{n}\sum_{j=1}^{n} \log a_j(x) =\frac{1}{n} \sum_{j\in
N_n(x)} \log a_j(x)=0 .
\end{eqnarray*}
Therefore we have proved $A \subset B $. For the inverse inclusion,
notice that
\begin{eqnarray*}
 \frac{1}{n} \log q_n(x) \leq \frac{1}{n} \sum_{j\in N_n(x)} \log (a_j(x)+1) + \frac{1}{n} \log q_{n-\sharp
 N_n}(1,\dots,1).
\end{eqnarray*}
Let $x\in B$. Since $B\subset C$, we have
\begin{eqnarray*}
 \lim_{n\to \infty} \frac{1}{n} \log q_{n-\sharp
 N_n}(1,\dots,1)= \frac{\gamma_0}{2}.
\end{eqnarray*}
Therefore by the assumption  $x\in B$, we get
\begin{eqnarray*}
\limsup_{n\to \infty}\frac{1}{n}\log q_n(x)\leq \frac{\gamma_0}{2}.
\end{eqnarray*}
Thus $B\subset A$.
\end{proof}

\subsection{Exponents $\gamma(x)$ and $\lambda(x)$}

In this subsection, we make a quick examination of the Khintchine
exponent $\gamma(x)$ and compare it with the Lyapunov exponent
$\lambda(x)$. Our main concern is the possible values of both
exponent functions.

A first observation is that for any $x\in [0,1)$, $\gamma(x)\geq 0 $
and $\lambda(x) \geq \gamma_0=2\log \frac{\sqrt{5}+1}{2} $. By the
Birkhoff ergodic theorem, we know that the Khintchine exponent
$\gamma(x)$ attains the value $\xi_0$ for almost all points $x$ with
respect to the Lebesgue measure. We will show that every positive
number is the Khintchine exponent $\gamma(x)$ of some point $x$.

\begin{pro}\label{p1}
For any $\xi \geq 0$, there exists a point $x_0\in [0,1)$ such that
$\gamma(x_0)=\xi$.
\end{pro}
\begin{proof} Assume $\xi>0$ ( for $\xi=0$, we take
$x_0=\frac{1+\sqrt{5}}{2}$ corresponding to $a_n\equiv 1$.) Take an
increasing sequence of integers $\{n_k\}_{k\geq 1}$ satisfying
$$n_0=1, \ \ n_{k+1}-n_k\to \infty, \ {\rm{and}}\  \frac{n_k}{n_{k+1}}\to 1,
{\rm{as}} \ \ k\to \infty.$$ Let $x_0\in (0,1)$ be a point whose
partial quotients satisfy
\begin{eqnarray*}
 e^{(n_k-n_{k-1})\xi}\leq a_{n_k}\leq
e^{(n_k-n_{k-1})\xi}+1; \ \ \ a_n=1 \ {\rm{otherwise}}.
\end{eqnarray*}
Since for $ n_k \leq n < n_{k+1}$,
\[ \frac{1}{n_{k+1}} \sum_{i=1}^{k}
\log e^{(n_i-n_{i-1})\xi} \leq \frac{1}{n}\sum_{j=1}^{n} \log a_j
\leq \frac{1}{n_k} \sum_{i=1}^{k} \log (e^{(n_i-n_{i-1})\xi}+1),
\]
we have
\[ \gamma(x_0)=\lim_{n\to \infty} \frac{1}{n}\sum_{j=1}^n \log a_j(x)= \xi. \]
\end{proof}

\smallskip
In the following, we will show that the set $E_\xi$ and $F_\lambda$
are never equal. So it is two different problems to study
$\gamma(x)$ and $\lambda(x)$. However, as we will see,
$E_{\xi}(\phi)=F_{2\xi}(\phi)$ when $\phi$ is faster than $n$.
\begin{pro}\label{l6}
For any $\xi\geq 0$ and $\lambda\geq 2\log \frac{\sqrt{5}+1}{2}$, we
have $E_\xi\neq F_\lambda$.
\end{pro}
$Proof.$ Given $\xi\ge 0$. It suffices to construct two numbers with
same Khintchine exponent $\xi$ but different Lyapunov exponents.

 For the first number, take just the number $x_0$ constructed in the proof of Proposition {\ref {p1}}. We
claim that
\begin{eqnarray}\label{112}
\lambda(x_0)=2\xi +2\log \frac{\sqrt{5}+1}{2}.
\end{eqnarray}
In fact, by Lemma \ref{l2.3} we have
\begin{equation}\label{113}
\prod_{j=1}^k(\frac{a_{n_j}+1}{2})q_{n_k-k}(1,\cdots, 1)\leq
q_{n_k}(a_1,\cdots, a_{n_k})\leq
\prod_{j=1}^k(a_{n_j}+1)q_{n_k-k}(1,\cdots, 1).
\end{equation}
Then by the assumption on $n_k$, we have
\begin{eqnarray*}
\lambda(x_0)=\lim_{n\to \infty}\frac{2}{n}\log q_n(x_0)=2(\xi+\log
\frac{\sqrt{5}+1}{2}).
\end{eqnarray*}

Construct now the second number.
 Fix $k\geq 1$. Define
$x_1=[\varsigma_1, \cdots, \varsigma_n, \cdots]$ where
$$\varsigma_n=\Bigg(\underbrace{\overbrace{1,\cdots, 1, \lfloor e^{k\xi}\rfloor}^k,
\cdots, 1,\cdots, 1, \lfloor e^{k\xi}\rfloor}_{kn},
\left\lfloor(\frac{e^{(k+1)\xi}}{[e^{k\xi}]})^n\right\rfloor\Bigg).$$
Notice that there are $n$ small vectors $(1,\cdots, 1, \lfloor
e^{k\xi}\rfloor)$ in $\varsigma_n$ and the length of $\varsigma_n$
is equal to $N_k:= kn+1$. We can prove
$$\gamma(x_1)=\xi, \ \ \lambda(x_1)=\lambda \left(\left[
\overline{1,\cdots, \lfloor
e^{k\xi}\rfloor}\right]\right)+2\xi-\frac{2}{k}\log \lfloor
e^{k\xi}\rfloor,$$ by  the same arguments as in proving the similar
result for $x_0$. It is clear that $\lambda(x_0) \neq \lambda(x_1)$
for large $k\geq 1$. \hfill{$\Box$}

It is evident that Proposition \ref{p1} and the formula (\ref{112})
yield the following result due to M. Pollicott and H. Weiss
\cite{PoW}.
\begin{cor}[\cite{PoW}]
For any $\lambda \geq 2\log \frac{\sqrt{5}+1}{2}$, there exists a
point $x_0\in [0,1)$ such that $\lambda(x_0)=\lambda$.
\end{cor}

\subsection{Pointwise dimension}
We are going to compare the pointwise dimension and the Markov
pointwise dimension (corresponding to continued fraction system) of
a Borel probability measure.

Let $\mu$ be a Borel probability measure on $[0,1)$. Define the
pointwise dimension and the Markov pointwise dimension respectively
by
$$
   d_\mu(x):= \lim_{r\to 0} \frac{\log \mu(B(x,r))}{\log r},
\ \
   \delta_\mu(x):= \lim_{n\to \infty} \frac{\log \mu(I_n(x))}{\log
   |I_n(x)|},
$$
if the limits exist, where $B(x,r)$ is the ball centered at $x$ with
radius $r$.

For two series $\{u_n\}_{n\geq 0}$ and $\{v_n\}_{n\geq 0}$, we write
$u_n\asymp v_n$ which means that there exist absolute positive
constants $c_1, c_2$ such that $c_1v_n\leq u_n\leq c_2 v_n$ for $n$
large enough. Sometimes, we need the following condition at a point
$x$:
\begin{eqnarray}\label{compare}
  \mu(B(x,|I_n(x)|))\asymp \mu(I_n(x)).
\end{eqnarray}
We have the following relationship between $\delta_\mu(x)$ and
$d_\mu(x)$.
\begin{lem}\label{d_delta}
Let $\mu$ be a Borel measure. \\(a) Assume (\ref{compare}).
  If $d_\mu(x)$ exists then $\delta_\mu (x)$exists and $\delta_\mu(x)=d_\mu(x)$.
 \\ (b) If $\delta_\mu (x)$ and $\lambda (x)$ both exist, then $d_\mu (x)$exists
  and $\delta_\mu(x)=d_\mu(x)$.
\end{lem}
\begin{proof} (a) \ If the limit defining $d_{\mu}(x)$ exists, then
the limit
\begin{eqnarray*}
\lim_{n\to +\infty}\frac{\log \mu(B(x, |I_n(x)|))}{\log |I_n(x)|}
\end{eqnarray*}
exists and equals to $d_{\mu}(x)$. Thus by (\ref{compare}), the
limit defining $\delta_{\mu}(x)$ also exists and equals to
$d_{\mu}(x)$.

 (b) \  Since $\lambda(x)$ exists, by Lemma {\ref{lemma2.5}} we have
\begin{eqnarray}\label{compare2}
\lim_{n\to \infty}\frac{\log |I_{n}(x)|}{\log
|I_{n+1}(x)|}=\lim_{n\to \infty}\frac{1}{n}\log
|I_n(x)|/\frac{1}{n+1}\log |I_{n+1}(x)|=1.
\end{eqnarray}
For any $r>0$, there exists an $n$ such that $|I_{n+1}(x)|\leq
r<|I_n(x)|$. Then by Lemma \ref{l2.5}, we have $I_{n+1}(x)\subset
B(x,r)\subset I_{n-2}(x)$. Thus
\begin{eqnarray}\label{compare3}
\frac{\log \mu(I_{n-2}(x))}{\log |I_{n+1}(x)|}\leq \frac{\log
\mu(B(x,r))}{\log r}\leq \frac{\log \mu(I_{n+1}(x))}{\log |I_n(x)|}.
\end{eqnarray}
Combining  (\ref{compare2}) and (\ref{compare3}) we get the desired
result.
\end{proof}

Let us give some measures for which the condition (\ref{compare}) is
satisfied. These measures will be used in the subsection \ref{5.1}.
The existence of these measures $\mu_{t,q}$ will be discussed in
Proposition \ref{measure_potential} and the subsection \ref{5.1}.
\begin{lem}\label{example-measure}
Suppose $\mu_{t,q}$ is a measure satisfying
$$\mu_{t,q}(I_n(x))\asymp \exp(-nP(t,q))|I_n(x)|^{t}\prod_{j=1}^{n}
a_j^q,$$ where $P(t,q)$ is a constant. Then (\ref{compare}) is
satisfied by $\mu_{t,q}$.
\end{lem}
$Proof.$ Notice that when $a_n(x)=1$,
$\mu_{t,q}(I_n(x))\asymp\mu_{t,q}(I_{n-1}(x))$.  Then in the light
of Lemma \ref{l2.5}, we can show that (\ref{compare}) is satisfied
by $\mu_{t,q}$. \hfill $\Box$


\setcounter{equation}{0}

\section{Fast growth rate: proof of Theorem \ref{thm-fast}}

\subsection{Lower bound}
We start with the mass distribution principle (see \cite{Fa1},
Proposition 4.2), which will be used to estimate the lower bound of
the Hausdorff dimension of a set.

\begin{lem}[\cite{Fa1}]
Let $E \subset [0,1)$ be a Borel set and $\mu$ be a measure with
$\mu(E)>0$. Suppose that
$$\liminf_{r\rightarrow 0} \frac{\log \mu(B(x,r))}{\log r} \geq s,
\quad \forall x\in E $$ where $B(x,r)$ denotes the open ball with
center at $x$ and radius $r$. Then $\dim E \geq s$.
\end{lem}

Next we  give  a formula for computing the Hausdorff dimension for a
class of Cantor sets related to continued fractions.

\begin{lem}\label{l3.1}
Let $\{s_n\}_{n \geq 1}$ be a sequence of positive integers tending
to infinity with $s_n\geq 3$ for all $n\geq 1$. Then for any
positive number $N \geq 2$, we have
\begin{eqnarray*}
\dim \{x\in [0,1): s_n \leq a_n(x) <N s_n\ \ \forall \ n \geq 1\} =
\liminf\limits_{n \to \infty}\frac{\log (s_1s_2\cdots
s_n)}{2\log(s_1s_2\cdots s_n)+\log s_{n+1}}.
\end{eqnarray*}
\end{lem}
\begin{proof} Let $F$ be the set in question and $s_0$ be the
$\liminf$ in the statement.
 We call $$J(a_1, a_2, \cdots, a_n):=Cl\bigcup_{a_{n+1}\geq s_{n+1}}I_{n+1}(a_1,\cdots, a_n,a_{n+1})$$  a {\it basic}
CF-{\it interval} of order $n$ with respect to $F$ (or simply basic
interval of order $n$), where $s_k \leq a_k <N s_k$ for all $1 \leq
k \leq n$. Here $Cl$ stands for the closure. Then it follows that
\begin{eqnarray}\label{f10} F=\bigcap_{n=1}^{\infty}\bigcup_{s_k\leq
a_k<Ns_k, 1\leq k\leq n}J(a_1,\cdots,a_n).
\end{eqnarray}
By Lemma \ref{length}, we have \begin{eqnarray}\label{f11}
J(a_1,\cdots,a_n)=\left[\frac{p_n}{q_n},
\frac{s_{n+1}p_n+p_{n-1}}{s_{n+1}q_n+q_{n-1}}\right]\ {\rm{or}}\ \
\left[
\frac{s_{n+1}p_n+p_{n-1}}{s_{n+1}q_n+q_{n-1}},\frac{p_n}{q_n}\right]
\end{eqnarray}according to $n$ is even or odd. Then by Lemma \ref{length}, Lemma \ref{l2.2}
and the assumption on $a_k$ that $s_k\leq a_k<Ns_k$ for all $1\leq
k\leq n$, we have
\begin{equation}\label{f3.12}
\frac{1}{2N^n}\frac{1}{s_{n+1}(s_1\cdots
s_n)^2}\leq\Big|J(a_1,\cdots,a_n)\Big|=\frac{1}{q_n(s_{n+1}q_n+q_{n-1})}\leq
\frac{1}{s_{n+1}(s_1\cdots s_n)^2}.
\end{equation}
Since $s_k\to \infty$ as $k\to \infty$, then\begin{eqnarray*}
\lim_{n\to \infty}\frac{\log s_1+\cdots+\log s_n}{n}=\infty.
\end{eqnarray*}
This, together with the definition of $s_0$, implies that for any
$s>s_0$, there exists a sequence $\{n_{\ell}: \ell\geq 1\}$ such
that for all $\ell\geq 1$,
$$ (N-1)^{n_{\ell}}<\left(s_{n_{\ell}+1}(s_1\cdots s_{n_{\ell}})^2\right)^{\frac{s-s_0}{2}},
\qquad \prod_{k=1}^{n_{\ell}}s_k\leq
\left(s_{{n_{\ell}}+1}(s_1\cdots
s_{n_{\ell}})^2\right)^{\frac{s+s_0}{2}}.
$$
Then, by (\ref{f10}), together with (\ref{f3.12}), we have
\begin{eqnarray*} H^s(F)&\leq &\liminf_{\ell\to \infty}\sum_{s_k\leq
a_k<Ns_k, 1\leq k\leq
n_{\ell}}\Big|J(a_1,\cdots,a_{n_{\ell}})\Big|^s\\
&\leq & \liminf_{\ell\to
\infty}\left((N-1)^{n_{\ell}}\prod_{k=1}^{n_{\ell}}s_k\right)\left(\frac{1}{s_{{n_{\ell}}+1}(s_1\cdots
s_{n_{\ell}})^2}\right)^s\leq 1.
\end{eqnarray*}
Since $s>s_0$ is arbitrary, we have $\dim F\leq s_0$.

 For the lower bound, we define a measure $\mu$ such that
for any basic $CF$-interval $J(a_1, a_2, \cdots, a_n)$ of order $n$,
$$
\mu(J(a_1, a_2, \cdots, a_n))=\prod_{j=1}^n \frac{1}{(N-1)s_j}.
$$
By the Kolmogorov extension theorem, $\mu$ can be extended to a
probability measure supported on $F$. In the following, we will
check  the mass distribution principle with this measure.

Fix $s<s_0$. By the definition of $s_0$ and the fact that $s_k\to
\infty \ (k\to \infty)$ and that $N$ is a constant, there exists an
integer $n_0$ such that for all $n\geq n_0$,
\begin{eqnarray}\label{f3.13} \prod_{k=1}^{n}(N-1)s_k\geq
\left(s_{n+1}(\prod_{k=1}^nNs_k)^2\right)^{s}.\end{eqnarray} We take
$\displaystyle r_0=\frac{1}{2N^{n_0}}\frac{1}{s_{n_0+1}(s_1\cdots
s_{n_0})^2}$. \vskip 3pt

For any $x\in F$, there exists an infinite sequence $\{a_1, a_2,
\cdots\}$ with $s_k\leq a_k<Ns_k, \forall k\geq 1$ such that $x\in
J(a_1,\cdots, a_n)$, for all $n\geq 1$. For any $r<r_0$, there
exists an integer $n\geq n_0$ such that $$ |J(a_1,\cdots,
a_{n+1})|\leq r<|J(a_1,\cdots, a_n)|.
$$

We claim that the ball $B(x,r)$ can intersect only one $n$-th basic
interval, which is just $J(a_1,\cdots, a_n)$. We establish this only
at the case that $n$ is even, since for the case that $n$ is odd,
the argument is similar.

Case (1): $s_n<a_n<Ns_n-1$. The left and right adjacent $n$-th order
basic intervals to $J(a_1,\cdots, a_n)$ are $J(a_1,\cdots, a_n-1)$
and $J(a_1,\cdots, a_n+1)$ respectively. Then by (\ref{f11}) and the
condition that $s_n\geq 3$,  the gap between $J(a_1,\cdots, a_n)$
and $J(a_1,\cdots, a_n-1)$ is \begin{eqnarray*} \ \ \
\frac{p_n}{q_n}-\frac{s_{n+1}(p_n-p_{n-1})+p_{n-1}}{s_{n+1}(q_n-q_{n-1})+q_{n-1}}
=\frac{s_{n+1}-1}{q_n\Big(s_{n+1}(q_n-q_{n-1})+q_{n-1}\Big)}\geq
\Big|J(a_1,\cdots, a_n)\Big|.
\end{eqnarray*}
Hence $B(x,r)$ can not intersect $J(a_1,\cdots, a_n-1)$. On the
other hand,
 the  gap  $J(a_1,\cdots, a_n)$
and $J(a_1,\cdots, a_n +1)$ is
\begin{eqnarray*}
\ \ \ \
\frac{p_n+p_{n-1}}{q_n+q_{n-1}}-\frac{s_{n+1}p_n+p_{n-1}}{s_{n+1}q_n+q_{n-1}}
=\frac{s_{n+1}-1}{(q_n+q_{n-1})(s_{n+1}q_n+q_{n-1})}\geq
\Big|J(a_1,\cdots, a_n)\Big|.
\end{eqnarray*}
Hence $B(x,r)$ can not intersect $J(a_1,\cdots, a_n+1)$ either.

Case (2): $a_n=s_n$. The right adjacent $n$-th order basic interval
to $J(a_1,\cdots, a_n)$ is $J(a_1,\cdots, a_n+1)$. The same argument
as in  the case (1) shows that  $B(x,r)$ can not intersect
$J(a_1,\cdots, a_n+1)$. On the other hand, the gap between the left
endpoint of $J(a_1,\cdots, a_n)$ and that of $I_{n-1}(a_1,\cdots,
a_{n-1})$ is
\begin{eqnarray*}
\frac{p_n}{q_n}-\frac{p_{n-1}+p_{n-2}}{q_{n-1}+q_{n-2}}=\frac{s_n-1}{(q_{n-1}+q_{n-2})q_n}\geq
\Big|J(a_1,\cdots, a_n)\Big|.
\end{eqnarray*} It follows that $B(x,r)$ can not intersect any
$n$-th order $CF$-basic intervals on the left of $J(a_1,\cdots,
a_n)$. In general, $B(x,r)$ can intersect no other $n$-th order
$CF$-basic intervals than $J(a_1,\cdots, a_n)$.

Case (3): $a_n=Ns_n-1$. From the case (1), we know that $B(x,r)$ can
not intersect any $n$-th order $CF$-basic intervals on the left of
$J(a_1,\cdots, a_n)$. While for on the right,  the gap between the
right endpoint of $J(a_1,\cdots, a_n)$ and that of
$I_{n-1}(a_1,\cdots, a_{n-1})$ is
\begin{eqnarray*}
\frac{p_{n-1}}{q_{n-1}}-\frac{s_{n+1}p_n+p_{n-1}}{s_{n+1}q_n+q_{n-1}}=\frac{s_{n+1}}{(s_{n+1}q_n+q_{n-1})q_{n-1
}}\geq \Big|J(a_1,\cdots, a_n)\Big|.
\end{eqnarray*}
It follows that $B(x,r)$ can not intersect any $n$-th order
$CF$-basic intervals on the right of $J(a_1,\cdots, a_n)$. In
general, $B(x,r)$ can intersect no other $n$-th order $CF$-basic
intervals than $J(a_1,\cdots, a_n)$.

Now we distinguish two cases to estimate the measure of $B(x,r)$.

$Case \ (i).$ \ $|J(a_1,\cdots, a_{n+1})|\leq r<|I_{n+1}(a_1,\cdots,
a_{n+1})|$. By Lemma \ref{l2.5} and the fact $a_{n+1}\neq 1$,
$B(x,r)$ can intersect at most five $(n+1)$-th order basic
intervals. As a consequence, by (\ref{f3.13}), we have
\begin{eqnarray} \mu(B(x,r))\leq
5\prod_{k=1}^{n+1}\frac{1}{(N-1)s_k}\leq
5\left(\frac{1}{s_{n+2}(N^{n+1}s_1\cdots s_{n+1})^2}\right)^{s}.
\end{eqnarray}
Since $$ r>\Big|J(a_1,\cdots,
a_{n+1})\Big|=\frac{1}{q_{n+1}(s_{n+2}q_{n+1}+q_n)}\geq
\frac{1}{2s_{n+2}(N^{n+1}s_1\cdots s_{n+1})^2},
$$
it follows that\begin{eqnarray*} \mu(B(x,r))\leq 10r^s.
\end{eqnarray*}

 $Case\ (ii).$ \ $|I_{n+1}(a_1,\cdots, a_{n+1})|\leq
r<|J(a_1,\cdots, a_{n})|$. In  this case, we have $$
I_{n+1}(a_1,\cdots, a_{n+1})=\frac{1}{q_{n+1}(q_{n+1}+q_n)}\geq
\frac{1}{2q_{n+1}^2}\geq
\frac{1}{2N^{2(n+1)}}\left(\prod_{k=1}^{n+1}s_{k}\right)^2.
$$
So $B(x,r)$ can intersect at most a number $8rN^{2(n+1)}(s_1\cdots
s_{n+1})^2$ of  $(n+1)$-th basic intervals. As a consequence,
\begin{eqnarray*}
\mu(B(x,r))&\leq& \min\Big\{\mu(J(a_1,\cdots,a_n)), \
8rN^{2(n+1)}(s_1\cdots
s_{n+1})^2\prod_{k=1}^{n+1}\frac{1}{(N-1)s_k}\Big\}\\
&\leq &\prod_{k=1}^{n}\frac{1}{(N-1)s_k}\min\Big\{1, \
8rN^{2(n+1)}(s_1\cdots s_{n+1})^2\frac{1}{(N-1)s_{n+1}}\Big\}.
\end{eqnarray*}
By (\ref{f3.13}) and the elementary inequality  $\min\{a, \ b\}\leq
a^{1-s}b^{s}$ which holds  for any $a, b>0$
 and $0<s<1$, we have\begin{eqnarray*}
\mu(B(x,r))&\leq& \left(\frac{1}{s_{n+1}(N^ns_1\cdots
s_n)^2}\right)^{s}\cdot \left(8rN^{2(n+1)}(s_1\cdots
s_{n+1})^2\frac{1}{(N-1)s_{n+1}}\right)^{s}\\
&\leq& 16N r^s.
 \end{eqnarray*}

 Combining these two cases, together with mass distribution
 principle, we have $\dim F\geq s_0$.\end{proof}

\smallskip
 Let $$E'=\{x\in [0,1):
e^{\varphi(n)-\varphi(n-1)}\leq a_n(x)\leq
2e^{\varphi(n)-\varphi(n-1)}, \ \forall n\geq 1\}.$$ It is evident
that $E'\subset E_\xi(\varphi)$. Then applying Lemma \ref{l3.1}, we
have\begin{eqnarray*} E_\xi(\varphi)\geq \liminf_{n\to
\infty}\frac{\varphi(n)}{\varphi(n+1)+\varphi(n)}=\frac{1}{b+1}.
\end{eqnarray*}

\subsection{Upper bound}
We first give a  lemma which is a little bit more than the upper
bound for the case $b=1$. Its proof uses a family of Bernoulli
measures with an infinite number of states.

\begin{lem}\label{l3.2}
If $\lim\limits_{n\to \infty}\frac{\varphi(n)}{n}=\infty$, then
$\dim E_{\xi}(\varphi)\leq \frac{1}{2}$.
\end{lem}
\begin{proof} For any $t>1$, we introduce a family of Bernoulli
 measures $\mu_t $:
\begin{eqnarray}
\mu_t(I_n(a_1, \cdots, a_n))=e^{-nC(t)-t\sum_{j=1}^{n}\log a_j(x)}
 \end{eqnarray}
 where $C(t)=\log \sum\limits_{n=1}^{\infty}\frac{1}{n^t}$.

 Fix $x\in E_\xi(\varphi)$ and $\epsilon >0$. If $n$ is sufficiently large, we have
  \begin{eqnarray}\label{120}
  (\xi-\epsilon)\varphi(n)<\sum_{j=1}^{n}\log a_j(x)<(\xi+\epsilon)\varphi(n).\end{eqnarray}
So \begin{eqnarray*} E_\xi(\varphi)\subset
\bigcap_{N=1}^{\infty}\bigcup_{n=N}^{\infty}E_n(\epsilon),
\end{eqnarray*}
where $$E_n(\epsilon)=\{x\in [0,1):(\xi-\epsilon)\varphi(n)<
\sum_{j=1}^{n}\log a_j(x)<(\xi+\epsilon)\varphi(n)\}.$$

Now let $\mathcal{I}(n, \epsilon)$ be the family of all $n$-th order
cylinders $I_n(a_1,\cdots, a_n)$ satisfying (\ref{120}). For each
$N\geq 1$, we select all those cylinders in
$\bigcup_{n=N}^{\infty}\mathcal{I}(n, \epsilon)$ which are maximal
($I \in \bigcup_{n=N}^{\infty}\mathcal{I}(n, \epsilon)$ is maximal
if there is no other $I'$ in $\bigcup_{n=N}^{\infty}\mathcal{I}(n,
\epsilon)$ such that $I\subset I'$ and $I\neq I'$). We denote by
$\mathcal{J}(N, \epsilon)$ the set of all maximal cylinders in
$\bigcup_{n=N}^{\infty}\mathcal{I}(n, \epsilon)$. It is evident that
$\mathcal{J}(N, \epsilon)$ is a cover of $E_{\xi}(\varphi)$. Let
$I_n(a_1,\cdots, a_n)\in \mathcal{J}(N, \epsilon)$, we have
\begin{eqnarray*}
\mu_t(I_n(a_1,\cdots, a_n))=e^{-nC(t)-t\sum\limits_{j=1}^n\log
a_j}\geq e^{-nC(t)-t(\xi+\epsilon)\varphi(n)}.
\end{eqnarray*}
On the other hand, \begin{eqnarray*} \Big|I_n(a_1,\cdots,
a_n)\Big|\leq e^{-2\log q_n}\leq e^{-2\sum\limits_{j=1}^n\log
a_j}\leq e^{-2(\xi -\epsilon)\varphi(n)}.
\end{eqnarray*}
Since $\lim\limits_{n\to \infty}\frac{\varphi(n)}{n}=\infty$, for
each $s> t/2$ and $N$ large enough, we have \begin{eqnarray*}
\Big|I_n(a_1,\cdots, a_n)\Big|^s\leq \mu_t(I_n(a_1,\cdots, a_n)).
\end{eqnarray*}
This implies  $\dim E_{\xi}(\varphi) \leq
1/2=\frac{1}{b+1}$.\end{proof}

\smallskip
Now we return back to the proof of the upper bound.

 {\it Case (i)}\
$b=1$. Since $\left(\varphi(n+1)-\varphi(n)\right) \uparrow \infty$,
Lemma \ref{l3.2} implies immediately $\dim E_{\xi}(\varphi)\leq
\frac{1}{2}.$

 {\it Case (ii)}\ $b>1$. By
(\ref{120}), for each $x\in E_\xi(\varphi)$ and $n$ sufficiently
large
$$
(\xi-\epsilon)\varphi(n+1)-(\xi+\epsilon)\varphi(n)\leq \log
a_{n+1}(x)\leq
(\xi+\epsilon)\varphi(n+1)-(\xi-\epsilon)\varphi(n).$$ Take
$$L_{n+1}=e^{(\xi-\epsilon)\varphi(n+1)-(\xi+\epsilon)\varphi(n)},\
\ M_{n+1}=e^{(\xi+\epsilon)\varphi(n+1)-(\xi-\epsilon)\varphi(n)}.$$
Define $$F_N=\{x\in [0,1]: L_n\leq a_n(x)\leq M_n, \forall n\geq
N\}.$$
 Then we have
\begin{eqnarray*}E_\xi(\varphi)\subset
\bigcup_{N=1}^{\infty}F_N.
\end{eqnarray*}We can only estimate the upper bound of $\dim F_1$.
Because  $F_N$ can be written as a countable union of sets with the
same form as $F_1$, then by the $\sigma$-stability of Hausdorff
dimension, we  will have $\dim F_N=\dim F_1$. We can further assume
that $M_n\geq L_n+2$.

For any $n\geq 1$, define $$ D_n=\{(\sigma_1,\cdots,\sigma_n)\in
\mathbb{N}^n: L_k\leq \sigma_k\leq M_k, \ 1\leq k\leq n\}.
$$
 It follows that$$
F_1=\bigcap_{n\geq 1}\bigcup_{(\sigma_1,\cdots,\sigma_n)\in
D_n}J(\sigma_1,\cdots,\sigma_n),
$$where $$
J(\sigma_1,\cdots,\sigma_n):=Cl \bigcup_{\sigma\geq
L_{n+1}}I(\sigma_1,\cdots,\sigma_n,\sigma)
$$ (called an admissible cylinder of order $n$).
For any $n\geq 1$ and $s>0$, we have\begin{eqnarray*}
\sum_{(\sigma_1,\cdots,\sigma_n)\in
D_n}\Big|J(\sigma_1,\cdots,\sigma_n)\Big|^s\leq
\sum_{(\sigma_1,\cdots,\sigma_n)\in
D_n}\Big|\frac{1}{q_n^2L_{n+1}}\Big|^s\leq \frac{M_1\cdots
M_n}{\Big((L_1\cdots L_n)^2L_{n+1}\Big)^s}.
\end{eqnarray*}
It follows that
\begin{eqnarray*} \dim F_1\leq \liminf_{n\to
\infty}\frac{\log M_1+\cdots+\log M_n}{\sum\limits_{k=1}^n\log
L_k+\sum\limits_{k=1}^{n+1}\log
L_k}=\frac{\xi+\epsilon+\frac{2\epsilon}{b-1}}{(\xi-\epsilon)(b+1)-2\epsilon-\frac{4\epsilon}{b-1}}.
\end{eqnarray*}
Letting $\epsilon \to 0$, we get \begin{eqnarray*}\dim
E_{\xi}(\varphi)\leq \frac{1}{b+1}.\end{eqnarray*}

\setcounter{equation}{0}

\section{Ruelle operator theory }
There have been various works on the Ruelle transfer operator for
the Gauss dynamics. See D. Mayer \cite{Ma1}, \cite{Ma2}, \cite{Ma3},
O. Jenkinson \cite{Jen}, O. Jenkinson and M. Pollicott \cite{JenP},
M. Pollicott and H. Weiss \cite{PoW}, P. Hanus, R. D. Mauldin and M.
Urbanski \cite{HMU}. In this section we will present a general
Ruelle operator theory for conformal infinite iterated function
system which was developed in \cite{HMU} and then apply it to the
Gauss dynamics. We will also prove some properties of the pressure
function in the case of Gauss dynamics ,  which will be used later.
\subsection{Conformal infinite iterated function systems}
In this subsection, we present the conformal infinite iterated
function systems which were studied by P. Hanus, R. D. Mauldin and
M. Urbanski in \cite{HMU}. See also the book of Mauldin and Urbanski
\cite{MUbook}.

 Let $X $ be a non-empty compact connected subset of $\mathbb{R}^d $ equipped with a metric $\rho $.
 Let $I $ be an index set with   at
 least two elements and at most countable elements.
An {\it{iterated function system}} $S=\{\phi_i: X\rightarrow X: i\in
I \}$ is a collection of injective contractions
 for which there exists $0<s<1 $ such that for each $i\in I $ and all $x, y \in X $,
 \begin{eqnarray}\label{contractive}
   \rho(\phi_i(x),\phi_i(y)) \leq s \rho(x,y).
\end{eqnarray}

 Before further discussion, we are willing to give a list of
 notation.\begin{eqnarray*}
&&\bullet \ I^n:=\{\omega : \omega=(\omega_1,\cdots,\omega_n),
\omega_k\in
 I, 1\leq k\leq n \},\\
 &&\bullet \ I^*:= \cup_{n\geq1} I^n,\\
&&\bullet\ I^{\infty}:= \Pi_{i=1}^{\infty} I,\\
&&\bullet\ \phi_{\omega}:=\phi_{\omega_1} \circ \phi_{\omega_2}
\circ \cdots \circ \phi_{\omega_n}, {\rm{for}} \
\omega=\omega_1\omega_2 \cdots \omega_n \in I^n,
n\geq 1,\\
&&\bullet\  |\omega |
\ {\rm{denote\ the\ length\ of\ }} \omega \in I^* \cup I^{\infty},\\
&&\bullet \ \omega |_n =\omega_1 \omega_2 \ldots \omega_n, {\rm{if}} \ \big|w\big|\geq n, \\
&&\bullet \ [\omega|_n] =[\omega_1\ldots \omega_n]=
 \{x \in I^{\infty}: \ x_1=\omega_1, \cdots, x_n=\omega_n \},\\
 &&\bullet \ \sigma:I^{\infty}\to I^{\infty} \ {\rm{the\ shift\
 transformation}},\\
 &&\bullet \ \|\phi'_{\omega}\|:=\sup_{x\in X} |\phi'_{\omega}(x)| \
 {\rm{for}}
\ \omega\in I^*,\\
&&\bullet \ C(X) \ {\rm{space \ of\ continuous \ functions\ on}}\ X, \\
&&\bullet \ ||\cdot ||_{\infty} \ {\rm{supremum\ norm\ on\ the\
Banach\ space\ }} C(X).\end{eqnarray*}

 For $\omega \in I^{\infty} $, the set
  \begin{eqnarray*}
   \pi(\omega)=\bigcap_{n=1}^{\infty} \phi_{\omega|_n} (X)
\end{eqnarray*}
is a singleton. We also denote its only element by $\pi(\omega) $.
This thus defines a coding map $\pi: I^{\infty} \rightarrow X $. The
limit set $J$ of the iterated function system is defined by
\begin{eqnarray*}
    J:= \pi (I^{\infty})
    .
\end{eqnarray*}

 Denote by $\partial X $ the boundary of $X$ and by $\mbox{\rm Int} (X) $ the interior of $X$.

 We say that the iterated function system $S=\{\phi_i\}_{i\in I}$ satisfies the {\it{open set condition}}
 if there exists a non-empty open set
 $U\subset X $ such that $\phi_i(U)\subset U $ for each $ i\in I$ and
 $\phi_i(U) \cap \phi_j(U)= \emptyset $ for each pair $i,j\in I, i \neq j $.

 An iterated function system $S=\{\phi_i :X\rightarrow X: i\in I \}$ is said to be {\it{conformal}}
 if the following are satisfied:\\
 \indent (1)\ the open set condition is satisfied for $U= \mbox{\rm Int}(X) $;\\
 \indent (2)\ there exists an open connected set $V$ with $X\subset V \subset \mathbb{R}^d $
 such that all maps
  $\phi_i$, $i\in I $, extend to $C^1 $ conformal diffeomorphisms of $ V$ into $V$;\\
 \indent (3)\ there exist $h,\ell >0 $ such that for each $x\in \partial X \subset \mathbb{R}^d$,
 there exists an open cone
 $\mbox{\rm Con}(x,h, \ell) \subset \mbox{\rm Int}(X) $ with vertex $x$, central angle of Lebesgue measure
 $h$ and altitude $\ell $;\\
 \indent (4) (Bounded Distortion Property) there exists $K\geq 1 $ such that
 $|\phi'_{\omega}(y)|\leq K |\phi'_{\omega}(x)| $
 for every $\omega \in I^* $ and every pair of points $x,y \in V $.

 The {\it{topological pressure function}} for a conformal iterated function systems
 $S=\{\phi_i :X\rightarrow X: i\in I \}$ is defined as
 \begin{eqnarray*}
    \mathcal{P}(t):= \lim_{n\to \infty} \frac{1}{n} \log \sum_{|\omega|=n} ||\phi'_{\omega}||^t.
\end{eqnarray*}
The system $S$ is said to be {\it regular} if there exists $t\geq 0
$ such that $\mathcal{P}(t)=0 $.

 Let $\beta >0 $. A {\it{H\"older family of functions}} of order $\beta$ is a family of continuous functions
 $F =\{ f^{(i)}: X \rightarrow \mathbb{C}: i\in I \} $ such that
\begin{eqnarray*}
    V_{\beta} (F) = \sup_{n\geq 1}  V_n(F)  < \infty,
\end{eqnarray*}
 where
 \begin{eqnarray*}
   V_n(F)= \sup_{\omega \in I^n} \sup_{x,y \in X}
   \{ |f^{(\omega_1)} (\phi_{\sigma(\omega)}(x)) - f^{(\omega_1)} (\phi_{\sigma(\omega)}(y))| \}
   e^{\beta(n-1)}.
\end{eqnarray*}

A family of functions $F=\{f^{(i)}: X\to \mathbb{R}, i\in I\}$ is
said to be {\it{strong}} if
\begin{eqnarray*}
    \sum_{i \in I} ||e^{f^{(i)}}||_{\infty} < \infty.
\end{eqnarray*}

Define the {\it Ruelle operator} on $C(X)$ associated to $F$ as
 \begin{eqnarray*}
    \mathcal L_F (g)(x) := \sum_{i\in I} e^{f^{(i)}(x)} g(\phi_i(x)).
 \end{eqnarray*}
Denote by $\mathcal L^*_F $ the dual operator of $\mathcal L_F$.

The {\it{topological pressure}} of $F$ is defined by
\begin{eqnarray*}
   P(F):=\lim_{n \to \infty} \frac{1}{n} \log \sum_{|\omega| =n} \exp \Bigg(\sup_{x\in X}
   \sum_{j=1}^{n} f^{\omega_j} \circ \phi_{\sigma^j\omega}(x) \Bigg).
\end{eqnarray*}

A measure $\nu $ is called  {\it{$F$-conformal}} if the following are satisfied:\\
\indent (1) $\nu$ is supported on $J$;\\
\indent (2) for any Borel set $A\subset X$ and any $\omega\in I^*$,
\begin{eqnarray*}
     \nu (\phi_{\omega}(A))= \int_A \exp \left( \sum_{j=1}^{n} f^{(\omega_j)} \circ \phi_{\sigma^j \omega}
     -P(F)|\omega| \right) d\nu;
\end{eqnarray*}
\indent (3) $\nu(\phi_{\omega}(X) \cap \phi_{\tau} (X))=0 \quad
  \omega,\tau\in I^n, \omega\neq \tau, n\geq 1. $

Two functions $\phi, \varphi\in C(X)$ are said to be {\it
cohomologous} with respect to the transformation $T$, if
 there exists $u\in C(X)$ such that $$
\varphi(x)=\phi(x)+u(x)-u(T(x)).
 $$

 The following two theorems are due to Hanus, Mauldin and
Urbanski \cite{HMU}.
\begin{thm}[\cite{HMU}]\label{Gibbs-Measure}
For a conformal iterated function system $S=\{\phi_i: X\rightarrow
X: i\in I \}$ and a strong H\"older family of functions $ F =\{
f^{(i)}: X \rightarrow \mathbb{C}: i\in I \} $, there exists a
unique $F$-conformal probability measure $\nu_F $  on $X$ such that
$\mathcal L^*_F \nu_F = e^{P(F)}\nu_F $. There exists  a unique
shift invariant probability measure $\tilde{\mu}_F $ on $I^{\infty}
$ such that $\mu_F:=\tilde{\mu}_F \circ \pi^{-1}$ is equivalent to
$\nu_F$ with bounded Radon-Nikodym derivative. Furthermore, the
Gibbs property is satisfied:
\begin{eqnarray*}
 \frac{1}{C} \leq \frac{\tilde{\mu}_F([\omega|_n])}
 {\exp \left( \sum_{j=1}^n f^{(\omega_j)} (\pi(\sigma^j \omega)) - nP(F)\right)}
 \leq C  .
\end{eqnarray*}
\end{thm}

Let $\Psi=\{\psi^{(i)}: X\rightarrow \mathbb{R}: i\in I \} $ and
  $ F=\{f^{(i)}: X\rightarrow \mathbb{R}: i\in I \}$ be two families of
   real-valued H\"older functions.
   We define the {\it{amalgamated functions}} on $I^{\infty}$ associated to $\Psi $ and $F$ as follows:
   \begin{eqnarray*}
   \tilde{\psi}(\omega):=\psi^{(\omega_1)}(\pi(\sigma\omega)), \qquad
     \tilde{f}(\omega):=f^{(\omega_1)}(\pi(\sigma\omega)) \qquad \forall \omega \in I^{\infty}.
\end{eqnarray*}

\begin{thm}[\cite{HMU}, see also \cite{MUbook}, pp. 43-48]\label{potential-analytic}
Let $\Psi$ and $F$ be two families of
   real-valued H\"older functions. Suppose the sets
   $\{i \in I: \sup_x (\psi^{(i)}(x)) >0 \}$ and $\{i \in I: \sup_x (f^{(i)}(x)) >0 \} $ are finite. Then the function
   $(t,q) \mapsto P(t,q)=P(t\Psi+qF)  $, is real-analytic with respect to $(t,q) \in \mbox{\rm Int} (D)  $, where
   $$D= \left\{ (t,q): \mathcal \sum_{i\in I } \exp ( \sup_x (t\psi^{(i)}(x)+qf^{(i)}(x))  ) < \infty \right\} .$$
   Furthermore, if $t\Psi+qF $ is a strong H\"older family for $(t,q) \in D $ and
   \begin{eqnarray*}
       \int (|\tilde{f}| +|\tilde{\psi}|) d\tilde{\mu}_{t,q} < \infty,
   \end{eqnarray*}
   where $\tilde{\mu}_{t,q}:=\tilde{\mu}_{t\Psi+qF} $ is obtained by
Theorem \ref{Gibbs-Measure},
   then
   \begin{eqnarray*}
     \frac{\partial P}{\partial t}= \int \tilde{\psi} d\tilde{\mu}_{t,q} \qquad {\rm{and}} \qquad
     \frac{\partial P}{\partial q}= \int \tilde{f} d\tilde{\mu}_{t,q}.
\end{eqnarray*}
   If \ $t\tilde{\psi}+q\tilde{f} $ is not cohomologous to a constant function,
   then $P(t,q) $ is strictly convex and
   \begin{eqnarray*}
H(t,q):= \left(
  \begin{array}{cccc}
    \frac{\partial^2 P}{\partial t^2} & \frac{\partial^2 P}{\partial t \partial q} \\
    \quad &\quad \\
    \frac{\partial^2 P}{\partial t \partial q} & \frac{\partial^2 P}{\partial q^2} \\
  \end{array}
\right)
\end{eqnarray*}
is positive definite.
\end{thm}

\subsection{Continued fraction dynamical system}
We apply the theory in the precedent subsection to the continued
fraction dynamical system. Let $X=[0,1]$ and $I=\mathbb{N} $. The
continued fraction dynamical system can be viewed as an iterated
function system:
$$ S=\left\{\psi_i(x)= \frac{1}{i+x}: i\in \mathbb{N} \right\} .$$
Recall that the projection mapping $\pi: I^{\infty} \rightarrow X $
is defined by
\begin{eqnarray*}
 \pi(\omega):=\bigcap_{n=1}^{\infty} \psi_{\omega|_n} (X), \ \forall
 \omega \in I^{\infty}.
\end{eqnarray*}

 Notice that $\psi'_1(0)=-1 $, thus
(\ref{contractive}) is not satisfied. However, this is not a real
problem, since we can consider the system of second level maps and
replace $S$ by $\tilde{S}:=\{ \psi_{i}\circ\psi_{j}: i,j\in
\mathbb{N} \}$. In fact,  for any $x\in [0,1)$\begin{eqnarray*}
(\psi_{i}\circ\psi_{j})'(x)=\Big(\frac{1}{i+\frac{1}{j+x}}\Big)'=\Big(\frac{1}{i(j+x)+1}\Big)^2\leq
\frac{1}{4}.
\end{eqnarray*}

In the following, we will collect or prove some facts on the
continued fraction dynamical system, which will be useful for
applying Theorem \ref{Gibbs-Measure} and \ref{potential-analytic}.

\begin{lem}[\cite{MU1}]\label{l4.3} The continued fraction dynamical
system $S$ is regular and conformal.
\end{lem}

For the investigation in the present paper, our problems are tightly
connected to the following two  families of H\"{o}lder functions.
$$\Psi=\{\log
|\psi'_i|: i\in \mathbb{N}\}\ {\rm{and}}\ F=\{-\log i: i\in
\mathbb{N}\}.
$$ \begin{rem}
We mention that our method used here is also applicable to other
potentials than  the two special families introduced here.
\end{rem}

The families $\Psi $ and $F$ are H\"older families and their
amalgamated functions are equal to
$$\tilde{\psi}(\omega)=-2\log (\omega_1+ \pi(\sigma\omega)),
\quad \tilde{f} (\omega)=-\log \omega_1  \quad \forall  \omega\in
\mathbb{N}^{\infty}. $$

For our convenience, we will consider the function $t\Psi-qF$
instead of $t\Psi+qF$.

\begin{lem}\label{l4.5}Let $D:=\{(t,q): 2t-q
> 1 \}.$
For any $(t,q)\in D$, we have

{\rm (i)} \ The family $t\Psi-qF:=\{t\log |\psi'_i|+q\log i: i\in
\mathbb{N}\} $ is H\"older and strong.

{\rm (ii)} \ The topological pressure $P$ associated to the
potential $t\Psi-qF$ can be written as
\begin{eqnarray*}
P(t,q)=\lim_{n\to \infty} \frac{1}{n} \log
\sum_{\omega_1,\cdots,\omega_n} \exp \Bigg(
  \sup_{x} \log \prod_{j=1}^{n}
\omega_j^q([\omega_j,\cdots, \omega_n+x])^{2t}\Bigg).
\end{eqnarray*}
\end{lem}
\begin{proof} The assertion on the domain $D$  follows from
\begin{eqnarray*} \frac{1}{4^t}\zeta(2t-q)=\mathcal
L_{t\Psi-qF}1=\sum_{i=1}^{\infty}\frac{i^q}{(i+x)^{2t}}\leq
\sum_{i=1}^{\infty}i^{q-2t}=\zeta(2t-q).
\end{eqnarray*}
where $\zeta(2t-q) $ is the  Riemann zeta function, defined by
\begin{eqnarray*}
   \zeta(s):=\sum_{n=1}^{\infty} \frac{1}{n^{s}} \quad \forall s>1.
\end{eqnarray*}

(i) \ For $(t,q)\in D$, write $(t\Psi- qF)^{(i)}:=t\log
|\psi'_i|+q\log i$. Then
\begin{eqnarray*}
  \sum_{i\in I} \Big\| \exp\left\{(t\Psi-qF)^{(i)}\right\}\Big\|_{\infty} =
  \sum_{i=1}^{\infty} \Big\|\frac{i^q}{(i+x)^{2t}}\Big\|_{\infty}
  =\sum_{i=1}^{\infty} i^{q-2t}= \zeta(2t-q)< \infty.
\end{eqnarray*}
Thus $t\Psi-qF$ is strong.

(ii) \ It suffices  to noticed that
\begin{eqnarray*}
   \sup_{x} \Bigg(
   \sum_{j=1}^{n} (t|\psi'_{\omega_j}|+q\log \omega_j) \circ \psi_{\sigma^j\omega}(x) \Bigg)
   =\sup_{x} \log \prod_{j=1}^{n}
\omega_j^q([\omega_j,\cdots, \omega_n+x])^{2t}.
\end{eqnarray*} \end{proof}

Denote by $\mathcal{L}_{t\Psi-qF}^* $ the conjugate operator of
$\mathcal{L}_{t\Psi-qF} $. Applying  Theorem \ref{Gibbs-Measure}
with the help of Lemma \ref{l4.3} and Lemma \ref{l4.5}, we get
\begin{pro}\label{measure_potential}
 For each $(t,q)\in D $, there exists a unique
 $t\Psi-qF $-conformal
probability measure $\nu_{t,q} $ on $[0,1] $ such that
$\mathcal{L}_{t\Psi-qF}^* \nu_{t,q} = e^{P(t,q)} \nu_{t,q} $, and a
unique shift invariant probability measure $\tilde{\mu}_{t,q} $ on
$\mathbb{N}^{\infty}$ such that $\mu_{t,q}:=\tilde{\mu}_{t,q} \circ
\pi^{-1} $ on $[0,1] $ is equivalent to $\nu_{t,q} $ and
\begin{eqnarray*}
 \frac{1}{C} \leq \frac{\tilde{\mu}_{t,q}([\omega|_n])}
 {\exp \left( \sum_{j=1}^n (t\Psi-qF)^{(\omega_j)} (\pi(\sigma^j \omega)) - nP(t,q)\right)}
 \leq C  \qquad \forall \omega \in \mathbb{N}^{\infty}.
\end{eqnarray*}
\end{pro}

\begin{lem}\label{l4.7}
For the amalgamated functions $\tilde{\psi}(\omega)=-2\log
(\omega_1+ \pi(\sigma\omega))$ and $\tilde{f} (\omega)=-\log
\omega_1$, we have\\
  \begin{eqnarray}\label{twoequ}
  \qquad  -\int \log|T'(x)| \mu_{t,q}= \int \tilde{\psi} d
    \tilde{\mu}_{t,q} \quad
{\rm{and}} \quad
 \int \log a_1(x) d \mu_{t,q}= -\int \tilde{f} d \tilde{\mu}_{t,q}.
\end{eqnarray}
and  $t\tilde{\psi}-q\tilde{f} $ is not cohomologous to a
constant.\end{lem}

\begin{proof} (i).\ Assertion (\ref{twoequ}) is just a consequence of the
facts
 $$-\log|T'(\pi(\omega))|=\tilde{\psi}(\omega), \qquad
\log a_1(\pi(\omega))=-\tilde{f}(\omega)  \quad \forall \omega\in
I^{\infty}.$$

 Suppose $t\tilde{\psi}-q\tilde{f} $ was not cohomologous to a
constant. Then there would be a
 bounded function $g $ such that $t\tilde{\psi}-q\tilde{f}=g- g\circ T + C
 $, which implies
 $$
   \lim_{n\to\infty} \frac{1}{n} \sum_{j=0}^{n-1}(t\tilde{\psi}-q\tilde{f})
   (\sigma^j\omega)
   =\lim_{n\to\infty} \frac{g- g\circ \sigma^n }{n} +C
   =C
 $$
 for all $\omega \in I^{\infty}$. On the other hand, if we take $\omega_1=[1,1,\cdots,]$,
 $\omega_2=[2,2,\cdots]$ and $\omega_3=[3,3,\cdots]$, we have
$$
   \lim_{n\to\infty} \frac{1}{n} \sum_{j=0}^{n-1} (t\tilde{\psi}-q\tilde{f})(\sigma^j\omega_i)
   =C_i,
 $$ where $$
 C_1=2t\log (\frac{\sqrt{5}-1}{2}),\ \  C_2=2t\log (\frac{\sqrt{5}-2}{2})+q\log
 2, \ \ C_3=2t\log (\frac{\sqrt{5}-3}{2})+q\log 3.$$
 Thus we get a contradiction.\end{proof}

 By Theorem \ref{potential-analytic} and
the proof of Lemma \ref{l4.5}, we know that $D= \{(t, q): 2t -q
>1\}$ is the analytic area of the pressure $P(t,q)$. Applying Lemma \ref{l4.7} and Theorem
\ref{potential-analytic}, we get more:
\begin{pro}\label{potential}
On $D= \{(t, q): 2t -q >1\}$,\\
\indent {\rm (1)} $P(t,q)$ is analytic, strictly convex.  \\
\indent {\rm (2)} $P(t,q)$ is strictly decreasing and strictly
convex with respect to $t$. In other words, $\frac{\partial
P}{\partial t}(t,q) < 0$
and $\frac{\partial^2 P}{\partial t^2}(t,q) > 0$. Furthermore,\\
\begin{eqnarray}\label{partial_t}
   \frac{\partial P}{\partial t} (t,q) = -\int \log|T'(x)| d \mu_{t,q}.
\end{eqnarray}
\indent {\rm (3)} $P(t,q)$ is strictly increasing and strictly
convex with respect to $q$. In other words, $\frac{\partial
P}{\partial q}(t,q)
> 0$ and $\frac{\partial^2 P}{\partial q^2}(t,q) > 0$. Furthermore,
\begin{eqnarray}\label{partial_q}
  \frac{\partial P}{\partial q} (t,q) = \int \log a_1(x) d \mu_{t,q}.
\end{eqnarray}
\indent {\rm (4)}
\begin{eqnarray*}
H(t,q):= \left(
  \begin{array}{cccc}
    \frac{\partial^2 P}{\partial t^2} & \frac{\partial^2 P}{\partial t \partial q} \\
    \quad &\quad \\
    \frac{\partial^2 P}{\partial t \partial q} & \frac{\partial^2 P}{\partial q^2} \\
  \end{array}
\right)
\end{eqnarray*}
is positive definite.
\end{pro}

\medskip

At the end of this subsection, we would like to quote some results
by D. Mayer \cite{Ma3} (see also M. Pollicott and H. Weiss
\cite{PoW}).

\begin{pro}[\cite{Ma3}]\label{p4.9}
Let $P(t):=P(t,0)$ and $\mu_t:=\mu_{t,0}$, then $P(t)$ is defined in
$(1/2,\infty)$ and we have $P(1)=0$ and $\mu_1=\mu_G$. Furthermore,
\begin{eqnarray}\label{P'(t)}
P'(t)= - \int \log |T'(x)| d\mu_t(x).
\end{eqnarray}
In particular
\begin{eqnarray}\label{P'(0)}
P'(0)= - \int \log |T'(x)| d\mu_G(x)=-\lambda_0.
\end{eqnarray}
\end{pro}

\begin{rem}\label{partial_q(1,0)}
Since $\mu_{1,0}=\mu_1= \mu_{G} $, by (\ref{partial_q}), we have
\begin{eqnarray}\label{P'(1,0)}
\frac{\partial P}{\partial q}(1,0)= \int \log a_1(x) d\mu_G= \xi_0.
\end{eqnarray}
\end{rem}

\subsection{Further study on $P(t,q) $}
We will use the following simple known fact of convex functions.
\begin{fac}\label{convex}
Suppose $f$ is a convex continuously differentiable function on an
interval $I$. Then $f'(x)$ is increasing and
\begin{eqnarray*}
f'(x) \leq \frac{f(y)-f(x)}{y-x} \leq f'(y)   \qquad x,y\in I, x<y.
\end{eqnarray*}
\end{fac}

First  we  give an estimation for the pressure $P(t,q)$ and show
some behaviors of $P(t,q)$ when $q$ tends to $-\infty$ and $2t-1$
($t$ being fixed).

\begin{pro}\label{p4.11}
\indent  For $(t,q)\in D$, we have
\begin{eqnarray}\label{pressure_ineq}
  -t \log 4 + \log \zeta(2t-q) \leq
P(t,q) \leq \log \zeta(2t-q).
\end{eqnarray}
Consequently,

 (1) \ $P(0,q)=\log \zeta(-q)$, and for any point $(t_0,q_0)$ on the line $2t-q=1$,
$$
\lim_{(t,q)\to (t_0,q_0)}P(t,q)=\infty;
$$

(2) for fixed $t\in \mathbb{R}$,
\begin{eqnarray}\label{part-q-infty}
\lim_{q\to 2t-1}\frac{\partial P}{\partial q}(t,q)=+\infty;
\end{eqnarray}

 (3) for fixed $t\in  \mathbb{R}$, we have
\begin{eqnarray}\label{P/q}
  \lim_{q\to -\infty} \frac{P(t,q)}{q}=0, \ \
\label{P/q-infty}
 \lim_{q\to -\infty}  \frac{\partial P}{\partial q}(t,q)=0.
\end{eqnarray}

\end{pro}
\begin{proof} Notice that
 $
  \frac{1}{\omega_j+1} \leq [\omega_j,\cdots, \omega_n+x] \leq
  \frac{1}{\omega_j}.
$ for  $x\in [0,1)$ and $1\leq j\leq n$. Thus we
have\begin{eqnarray*} \frac{1}{4^{nt}}
\sum_{\omega=1}^{\infty}(\omega^{q-2t})^n\leq
\sum_{\omega_1,\cdots,\omega_n}\prod_{j=1}^n\omega_j^q[\omega_j,\cdots,\omega_n+x]^{2t}\leq
\sum_{\omega=1}^{\infty}(\omega^{q-2t})^n.
\end{eqnarray*}
Hence by Lemma \ref{l4.5} (ii), we get (\ref{pressure_ineq}).

We get (1) immediately from (\ref{pressure_ineq}).

Look at (2). For all $q> q_0$, by the convexity of $P(t,q) $ and
Fact \ref{convex}, we have
\begin{eqnarray*}
\frac{\partial P}{\partial q}(t,q) \geq \frac{P(t, q)-
P(t,q_0)}{q-q_0}.
\end{eqnarray*}
Thus
\begin{eqnarray*}
\lim_{q\to 2t-1} \frac{\partial P}{\partial q}(t,q) \geq \lim_{q\to
2t-1}\frac{P(t, q_0)- P(t,q)}{q_0-q}=\infty.
\end{eqnarray*}
Here we use the fact that $\lim\limits_{q\to 2t-1} P(t,q)= +\infty
$. Hence we get (\ref{pressure_ineq}).

 In order to show (3), we consider
$P(t,q)/q $ as function of $q$ on $(-\infty,2t-1)\setminus \{0\}$.
Noticed that for fixed $t\in \mathbb{R}$, $\lim_{q\to -\infty}
\zeta(2t-q)=1$. Thus
\[
  \lim_{q\to -\infty} \frac{\log \zeta(2t-q)}{q}=0.
\]
Then the first formula in (\ref{P/q}) is followed from
(\ref{pressure_ineq}).

Fix $q_0< 2t-1$. Then for all $q< q_0$, by the convexity of $P(t,q)
$ and Fact \ref{convex}, we have
 \begin{eqnarray*}
\frac{\partial P}{\partial q}(t,q) \leq \frac{P(t, q_0)-
P(t,q)}{q_0-q}.
\end{eqnarray*}
Thus
\begin{eqnarray*}
\lim_{q\to -\infty} \frac{\partial P}{\partial q}(t,q) \leq
\lim_{q\to -\infty}\frac{P(t, q_0)- P(t,q)}{q_0-q}=0.
\end{eqnarray*}
Hence by Proposition \ref{potential} (3), we get the second formula
in (\ref{P/q-infty}). \end{proof}

\subsection{Properties of $(t(\xi), q(\xi))$}

 Recall that $\xi_0= \int \log a_1(x) \mu_G $ and $D_0:=
\{(t,q): 2t-q>1, 0\leq t \leq 1 \} $.

\begin{pro}\label{solution}
 For any $\xi\in (0, \infty)$, the system
\begin{eqnarray}\label{equations}
\left\{
  \begin{array}{ll}
     P(t,q)=q\xi, \\
     \displaystyle \frac{\partial P}{\partial q}(t,q)= \xi
  \end{array}
\right.
\end{eqnarray}
admits a unique solution $(t(\xi),q(\xi))\in D_0$. For $\xi=\xi_0 $,
the solution is $(t(\xi_0),q(\xi_0))=(1,0) $. The function $t(\xi)$
and $q(\xi)$ are analytic.

\end{pro}
\begin{proof}  $Existence \ and \ uniqueness \ of \ solution$ $(t(\xi),
q(\xi))$. \ \ \ Recall that $P(1,0)=0 $ and $P(0,q)= \log \zeta(-q)
$ (Proposition \ref{p4.11}).

We start with a geometric argument which will followed by a rigorous
proof. Consider $P(t,q) $ as a family of function of $q$ with
parameter $t$. It can be seen from the graph (see Figure 3) that for
any $\xi
> 0 $, there exists a unique $t\in (0,1] $, such that the line $\xi
q $ is tangent to $P(t,\cdot) $. This $t=t(\xi) $ can be described
as the unique point such that
\begin{eqnarray}\label{describ-t}
\inf _{q<2t(\xi)-1} \Big(P(t(\xi),q)-q \xi\Big) =0.
\end{eqnarray}
We denote by $q(\xi) $ the point where the infimum in
(\ref{describ-t}) is attained. Then the tangent point is
$(q(\xi),P(t(\xi),q(\xi)))$ and the derivative of $P(t(\xi),q)-q \xi
$ (with respect to $q$) at $q(\xi)$ equals $0$, i.e.,
\begin{eqnarray*}
 \Big(P(t(\xi),q)-q \xi \Big)'\arrowvert_{q(\xi)} = 0.
\end{eqnarray*}
Thus we have $\frac{\partial P}{\partial q}(t(\xi),q(\xi))= \xi $.
By (\ref{describ-t}), we also have $P(t(\xi),q(\xi))-q(\xi) \xi=0 $.
Therefore $(t(\xi),q(\xi)) $ is a solution of (\ref{equations}). The
uniqueness of $q(\xi)$ follows by the fact that $\frac{\partial
P}{\partial q}$ is monotonic with respect to $q$ (Proposition
\ref{potential}).
\begin{center}
\begin{pspicture}(0,0)(10,7)
\rput(5,3.3){\includegraphics{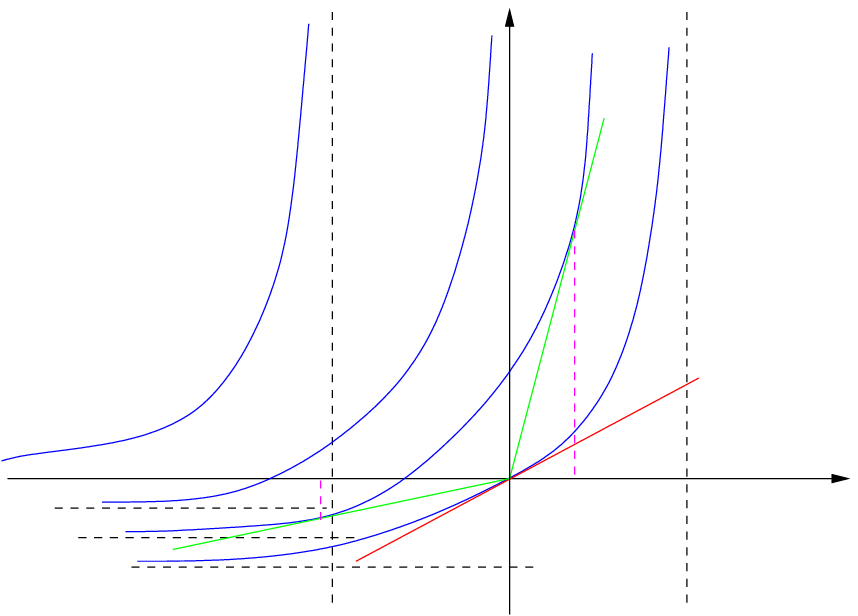}} \rput(9.15,1.35){$q$}
\rput(5.82,6.65){$P(t,q)$} \rput(8,5.6){$t=1$}
\rput(6.68,6.03){$t=t(\xi)$} \rput(5.1,5.7){$t=\frac{1}{2}$}
\rput(3.1,4.6){$t=0$} \rput(6.5,1.35){$q(\xi)$}
\rput(7.9,2.65){$\xi_0$} \rput(5.7,1.76){$0$} \rput(6.9,5.3){$\xi$}
\end{pspicture}
\begin{center}
{{\it Figure 3.}  Solution of (\ref{equations})}
\end{center}
\end{center}

Let us give a rigorous proof. By (\ref{part-q-infty}), (\ref{P/q})
and the mean-value theorem, for fixed $t\in \mathbb{R}$ and any
$\xi>0$, there exists a $q(t,\xi)\in (-\infty, 2t-1)$ such that
\begin{eqnarray}\label{2equation}
   \frac{\partial P}{\partial q}\big(t,q(t,\xi)\big)=\xi.
\end{eqnarray}
The monotonicity of $\frac{\partial P}{\partial q}$ with respect to
$q$ implies the uniqueness of $q(t,\xi)$ (Proposition
\ref{potential}).

Since $P(t,q)$ is analytic, the implicit $q(t,\xi) $ is analytic
with respect to $t $ and $\xi$. Fix $\xi$ and set $$W(t):=
P\big(t,q(t,\xi)\big)- \xi q(t,\xi).$$ Since
\begin{eqnarray*}
  W'(t) &=& \frac{\partial P}{\partial t}\big(t,q(t,\xi)\big)+\frac{\partial P}{\partial q}
  \big(t,q(t,\xi)\big)
  \frac{\partial q}{\partial t}(t,\xi) -\xi \frac{\partial q}{\partial t}(t,\xi)\\&
  =&
   \frac{\partial P}{\partial t}\big(t,q(t,\xi)\big)  \ \ \   \ \ \ \ \ ({\rm{by}}
   (\ref{2equation}))\\
   &<&0 \ \ \   \ \ \ \ \ \ \ \ \ \ \ \ \ \ ({\rm{by \ Proposition \ }} \ref{potential} (2)).
\end{eqnarray*}
Thus $W(t)$ is strictly decreasing.

Since $P(0,q)= \log \zeta(-q) >0$ ($q<-1$),  for $\xi>0 $ we have
\[W(0)=P\big(0, q(0, \xi)\big)- \xi q(0,\xi)
>0.\]
Since $P(1,q)$ is convex and $P(1,0)=0$, by Fact \ref{convex} we
have
\begin{eqnarray*}
\frac{P\big(1,q(1,\xi)\big)-0}{q(1,\xi)-0}\leq \frac{\partial
P}{\partial q}\big(1,q(1,\xi)\big) = \xi, \qquad {\rm if } \ \
q(1,\xi)>0,
\end{eqnarray*}
and
\begin{eqnarray*}
\frac{0-P\big(1,q(1,\xi)\big)}{0-q(1,\xi)}\geq \frac{\partial
P}{\partial q}\big(1,q(1,\xi)\big) = \xi, \qquad {\rm if } \ \
q(1,\xi)<0.
\end{eqnarray*}
If $q(1,\xi)=0$, we have in fact $\xi=\xi_0$ and
$P\big(1,q(1,\xi)\big)=0 $. Hence, in any case we have
\begin{eqnarray}\label{P(1,q)}
P\big(1,q(1,\xi)\big)- \xi q(1,\xi) \leq 0.
\end{eqnarray}
Therefore, $W(1)=P\big(1,q(1,\xi)\big)-\xi q(1,\xi)\leq 0 $.

Thus by the mean-value theorem and the monotonicity of $W(t)$, there
exists a unique $t=t(\xi)\in (0,1] $ such that $W(t(\xi))= 0 $,
{\it{i.e.}}
\begin{eqnarray}\label{f4.27}
 P\Big(t(\xi),q\big(t(\xi),\xi\big)\Big)=\xi q\big(t(\xi),\xi\big).
\end{eqnarray}
If we write $q\big(t(\xi),\xi\big) $ as $q(\xi)$, both
(\ref{2equation}) and (\ref{f4.27}) show that
$\big(t(\xi),q(\xi)\big) $ is the unique solution of
(\ref{equations}). For $\xi=\xi_0$, the assertion in Proposition
\ref{p4.9} that $P(0,1)=0=0\cdot \xi_0$ and the assertion of Remark
\ref{partial_q(1,0)} that $\frac{\partial P}{\partial q}(1,0)=\xi_0$
imply that $(0,1)$ is a solution of (\ref{equations}). Then the
uniqueness of the solution to (\ref{equations}) implies
$\big(t(\xi_0), q(\xi_0)\big)=(0,1)$.

$Analyticity \ of \ \big(t(\xi), q(\xi)\big)$. \ \ \ Consider the
map
\begin{eqnarray*}
F=\left(
    \begin{array}{c}
      F_1 \\
      F_2 \\
    \end{array}
  \right)=\left(
            \begin{array}{c}
              P(t,q)-q\xi \\
              \frac{\partial P}{\partial q}(t,q)-\xi \\
            \end{array}
          \right).
\end{eqnarray*}
Then the jacobian of $F$ is equal to \begin{eqnarray*} J(F)=:\left(
                 \begin{array}{cc}
                   \frac{\partial F_1}{\partial t} & \frac{\partial F_1}{\partial q} \\
                     \frac{\partial F_2}{\partial t} & \frac{\partial F_1}{\partial q} \\
                      \end{array}
                      \right)=\left(
                      \begin{array}{cc}
                       \frac{\partial P}{\partial t} & \frac{\partial P}{\partial q}-\xi \\
                        \frac{\partial^2 P}{\partial t \partial q} & \frac{\partial^2 P}{\partial q^2} \\
                        \end{array}
                        \right).
\end{eqnarray*}
Consequently,\begin{eqnarray*} \det (J(F))|_{t=t(\xi),
q=q(\xi)}=\frac{\partial P}{\partial t}\cdot \frac{\partial^2
P}{\partial q^2}\neq 0.
\end{eqnarray*}
Thus by the implicit function theorem, $t(\xi) $ and $q(\xi) $ are
analytic.\end{proof}

\smallskip
Now let us present some  properties on $t(\xi)$. Recall that
$\xi_0=\frac{\partial P}{\partial q}(1,0) $.

\begin{pro}\label{sign-q-xi}
   $q(\xi) < 0$ for $\xi< \xi_0$; $q(\xi_0) = 0$;  $ q(\xi) > 0$ for $\xi> \xi_0$.
\end{pro}
\begin{proof} Since $P(1,q)$ is convex and $P(1,0)=0$, by Fact \ref{convex},
we have
\begin{eqnarray*}
\frac{P(1,q)-0}{q-0}\geq \frac{\partial P}{\partial q}(1,0) = \xi_0,
\ \ (q>0); \qquad \frac{0-P(1,q)}{0-q}\leq \frac{\partial
P}{\partial q}(1,0) = \xi_0, \ \ (q<0).
\end{eqnarray*}
Hence for all $q<1$,
\begin{eqnarray}\label{P(1,q)q}
P(1,q) \geq \xi_0 q.
\end{eqnarray}

 We recall that $(t(\xi_0), q(\xi_0))= (1,0) $ is the
unique solution of the system (\ref{equations}) for $\xi =\xi_0 $.
By the above discussion of the existence of $t(\xi)$,
 $t(\xi)=1$ if and only if $\xi=\xi_0 $. Now we suppose $t\in (0,1) $.
For $\xi>\xi_0 $, using (\ref{P(1,q)q}), we have
\begin{eqnarray*}
   P(t,q)> P(1,q) \geq q \xi_0 \geq q \xi \qquad (\forall q\leq 0).
\end{eqnarray*}
Thus $q(\xi)> 0 $.  For $\xi<\xi_0 $, using (\ref{P(1,q)q}), we have
\begin{eqnarray*}
   P(t,q)> P(1,q) \geq q \xi_0 \geq q \xi \qquad (\forall q\geq 0).
\end{eqnarray*}
Thus $q(\xi)< 0 $. \end{proof}

\begin{pro}\label{sign-t'}
For $\xi \in (0, +\infty) $, we have
\begin{eqnarray}\label{t'}
    t'(\xi)= \frac{q(\xi)}{\frac{\partial P}{\partial t}(t(\xi),q(\xi))}.
\end{eqnarray}
\end{pro}
\begin{proof} Recall that
\begin{eqnarray}\label{equations_xi}
\left\{
  \begin{array}{ll}
     P(t(\xi),q(\xi))=q(\xi)\xi, \\
     \displaystyle \frac{\partial P}{\partial q}(t(\xi),q(\xi))=
     \xi.
  \end{array}
\right.
\end{eqnarray}
By taking the derivation with respect to $\xi$ of the first equation
in (\ref{equations_xi}), we get
\begin{eqnarray*}
  t'(\xi)\frac{\partial P}{\partial t}(t(\xi),q(\xi)) +
  q'(\xi)\frac{\partial P}{\partial q}(t(\xi),q(\xi)) = q'(\xi) \xi + q(\xi).
\end{eqnarray*}
Taking into account the second equation in (\ref{equations_xi}), we
get
\begin{eqnarray}\label{E1_1}
  t'(\xi)\frac{\partial P}{\partial t}(t(\xi),q(\xi)) =  q(\xi).
\end{eqnarray}
\end{proof}

\begin{pro}
We have $t'(\xi) > 0$ for $\xi< \xi_0$, $t'(\xi_0) = 0$,  and
$t'(\xi) < 0$ for $\xi> \xi_0$. Furthermore,
\begin{eqnarray}\label{t-xi-0}
  t(\xi) \to 0 \quad  (\xi \to 0),
\end{eqnarray}
\begin{eqnarray}\label{t-xi-infty}
 t(\xi) \to 1/2   \quad  (\xi \to +\infty).
\end{eqnarray}
\end{pro}

\noindent $Proof.$
 By Propositions \ref{sign-q-xi} and \ref{sign-t'} and the
fact $\frac{\partial P}{\partial t}>0$, $t(\xi) $ is increasing on
$(0, \xi_0) $ and decreasing on $(\xi_0, \infty) $. Then by the
analyticity of $t(\xi) $, we can obtain two analytic inverse
functions on the two intervals respectively. For the first inverse
function, write $\xi_1=\xi_1(t)$. Then $\xi'_1(t)>0$ and
\begin{eqnarray*}
  \xi_1(t) = \frac{P(t,q(t))}{q(t)}= \frac{\partial P}{\partial
  q}(t,q(t)).
\end{eqnarray*}
(the equations (\ref{equations}) are considered as equations on
$t$).
 By Proposition \ref{sign-q-xi}, we have
$q(\xi_1(t))<0$ then $P(t,q(\xi_1(t)))<0$. By Proposition
\ref{p4.11} (1), $\lim\limits_{q\to 2t-1}P(t,q)=\infty$. Thus there
exists $q_0(t)$ such that $q_0(t)>q(t)$ and $P(t,q_0(t))=0$.
Therefore
\begin{eqnarray*}
 \xi_1(t) =  \frac{\partial P}{\partial q}(t,q(t))< \frac{\partial P}{\partial
   q}(t,q_0(t)).
\end{eqnarray*}
Since $P(0,q)=\log \zeta(-q)$, we have $\lim_{t\to 0} q_0(t) =
\-\infty$. Thus we get
\[\lim_{t\to 0}\frac{\partial
P}{\partial q}(t,q_0(t))= \lim_{q\to -\infty} \frac{\partial
P}{\partial q}(0,q)=0.\]
  Hence by $\xi_1(t)\geq 0$, we obtain $
\lim_{t\to 0}\xi_1(t) =0$ which implies (\ref{t-xi-0}).

Write $\xi_2=\xi_2(t) $ for the second inverse function. Then
$\xi'_2(t)<0 $ and
\begin{eqnarray*}
    \xi_2(t) = \frac{P(t,q(t))}{q(t)} =  \frac{\partial P}{\partial q}(t,q(t))> \frac{\partial P}{\partial q}(t,0)
    \to \infty \quad (t \to 1/2).
 \end{eqnarray*}
This implies (\ref{t-xi-infty}). \hfill{$\Box$}

Let us summarize.   We have proved that $t(\xi)$ is analytic on $(0,
\infty) $, $\lim\limits_{\xi\to 0} t(\xi)=0 $ and
$\lim\limits_{\xi\to \infty} t(\xi)=1/2 $. We have also proved that
$t(\xi) $ is increasing on $(0,\xi_0) $, decreasing on
$(\xi_0,\infty) $ and $t(\xi_0)=1 $.

\setcounter{equation}{0}

\section{ Khintchine spectrum}
Now we are ready to study the Hausdorff dimensions of the level set
$$
E_{\xi}=\{x\in [0,1):\lim_{n\to \infty} \frac{1}{n} \sum_{j=1}^{n}
\log a_j( x) =\xi\}.
$$
Since $\mathbb{Q}$ is countable, we need only to consider
$$
\{x\in [0,1)\setminus \mathbb{Q} :\lim_{n\to \infty} \frac{1}{n}
\sum_{j=1}^{n} \log a_j( x) =\xi\}.
$$
which admits the same Hausdorff dimension with $E_{\xi} $ and is
still denoted by $E_{\xi}$.

\subsection{Proof of Theorem 1.2 (1) and (2)}\label{5.1}
 Let $(t,q)\in D $
 and $\mu_{t,q},\tilde{\mu}_{t,q}$  be the measures in Proposition \ref{measure_potential}.
For $x\in [0,1)\setminus \mathbb{Q}$, let
$x=[a_1,\cdots,a_n,\cdots]$ and $\omega=\pi^{-1}(x)$. Then
$\omega=a_1\cdots a_n\cdots\in \mathbb{N}^\mathbb{N}$ and
$$
 \mu_{t,q}(I_n(x))=\mu_{t,q}(I_n(a_1,\cdots,a_n))=\tilde{\mu}_{t,q}([\omega|_n]).
$$
By the Gibbs property of $\tilde{\mu}_{t,q}$,
\begin{eqnarray*}
\tilde{\mu}_{t,q}(\pi([\omega|_n]))&\asymp&
\exp(-nP(t,q))\prod_{j=1}^{n} \omega_j^q(\omega_j+\pi(\sigma^j
\omega))^{-2t}.
\end{eqnarray*}
In other words,
\begin{eqnarray*}
\mu_{t,q}(I_n(x))\asymp  \exp(-nP(t,q))\prod_{j=1}^{n}
a_{j}^q[a_j,\cdots,a_n,\cdots]^{2t}.
\end{eqnarray*}
By Lemma \ref{lemma2.5}, $|I_n(x)|\asymp
|(T^n)'(x)|^{-1}=\prod_{j=0}^{n-1}|T^j(x)|^2$. Thus we have the
following Gibbs property of $\mu_{t, q}$:
\begin{eqnarray}\label{estimate}
\mu_{t,q}(I_n(x))\asymp  \exp(-nP(t,q))|I_n(x)|^{t}\prod_{j=1}^{n}
a_j^q.
\end{eqnarray}

It follows that
\begin{eqnarray*}
 \delta_{\mu_{t,q}} (x)= \lim_{n\to \infty} \frac{\log
\mu_{t,q}(I_n(x))}{\log |I_n(x)|} = t + \lim_{n \to \infty} \frac{q
\cdot \frac{1}{n} \sum_{j=1}^{n} \log a_j - P(t,q)}{\frac{1}{n}\log
|I_n(x)|}.
\end{eqnarray*}
The Gibbs property of
 $\tilde{\mu}_{t,q}$ implies that $\mu_{t,q}$ is ergodic. Therefore,
$$
 \delta_{\mu_{t,q}} (x) = t + \frac{q \int \log a_1(x) d \mu_{t,q} -
P(t,q)}{- \int \log|T'(x)| d \mu_{t,q}}  \ \ \ \ \ \mu_{t,q}-a.e..
$$
Using the formula (\ref{partial_t}) and (\ref{partial_q}) in
Proposition \ref{potential}, we have
\begin{eqnarray}
\delta_{\mu_{t,q}} (x) = t + \frac{q \frac{\partial P}{\partial q}
(t,q) - P(t,q) }{\frac{\partial P}{\partial t} (t,q)}\ \  \ \
\mu_{t,q} -a.e..
\end{eqnarray}
Moreover, the ergodicity of $\tilde{\mu}_{t,q}$ also implies that
the Lyapunov exponents $\lambda(x)$ exist for $\mu_{t,q}$ almost
every $x$ in $[0,1)$. Thus by (\ref{estimate}), Lemma \ref{d_delta}
and Lemma \ref{example-measure}, we obtain
 \begin{eqnarray}\label{123}
d_{\mu_{t,q}}(x)= \delta_{\mu_{t,q}} (x) = t + \frac{q
\frac{\partial P}{\partial q} (t,q) - P(t,q) }{\frac{\partial
P}{\partial t} (t,q)}\ \ \ \ \mu_{t,q} -a.e..
\end{eqnarray}
For $\xi \in (0, \infty)$, choose $(t,q)=(t(\xi),q(\xi))\in {D_0}$
 be the unique solution of (\ref{equations}).
Then (\ref{123}) gives
\begin{eqnarray*}
   d_{\mu_{t(\xi),q(\xi)}}(x) =t(\xi)\ \ \ \ \mu_{t,q}-a.e..
   \end{eqnarray*}
By the ergodicity of $\tilde{\mu}_{t(\xi),q(\xi)}$ and
(\ref{partial_q}), we have for $\mu_{t(\xi),q(\xi)}$ almost every
$x$,
\begin{eqnarray*}
\lim_{n\to \infty} \frac{1}{n}\sum_{j=1}^{n} \log a_j (x)= \int \log
a_1(x) d \mu_{t(\xi),q(\xi)} = \frac{\partial P}{\partial
q}(t(\xi),q(\xi))=\xi.
\end{eqnarray*}
So $\mu_{t(\xi),q(\xi)} $ is supported on $E_\xi$. Hence
\begin{eqnarray}\label{lowbound}
   \dim (E_\xi) \geq \dim \mu_{t(\xi),q(\xi)} =t(\xi).
\end{eqnarray}

In the following we will show that
\begin{eqnarray}\label{dim-less-t}
\dim (E_\xi) \leq t \qquad (\forall t> t(\xi)).
\end{eqnarray}
Then it will imply that $\dim (E_{\xi})=t(\xi) $ for any $\xi>0$.
 For any $t>t(\xi)$, take an $\epsilon_0>0$ such that
$$
   0<\epsilon_0 < \frac {P(t(\xi),q(\xi))-P(t,q(\xi))}{q(\xi)} \quad {\rm{if}} \
   q(\xi)>0,
$$
and
$$
   0<\epsilon_0 < \frac {P(t,q(\xi))-P(t(\xi),q(\xi))}{q(\xi)} \quad {\rm{if}} \ q(\xi)<0 .
$$
(For the special case $q(\xi)=0 $, {\it{i.e.}}, $\xi=\xi_0$, we have
$\dim E_{\xi}=1 $ which is a well-known result). Such an
$\epsilon_0$ exists, for $P(t,q)$ is strictly decreasing with
respect to $t$. For all $n \geq 1$, set
$$
  E_{\xi}^n(\epsilon_0):= \bigg\{x \in [0,1)\setminus \mathbb{Q} : \xi -\epsilon_0 < \frac{1}{n}
\sum_{j=1}^{n} \log a_j( x) < \xi+ \epsilon_0 \bigg\}.
$$
Then we have
$$
    E_{\xi} \subset \bigcup_{N=1}^{\infty} \bigcap_{n=N}^{\infty }
E_{\xi}^n(\epsilon_0).
$$
Let $\mathcal{I}(n,\xi,\epsilon_0)$ be the collection of all $n$-th
order cylinders $I_n(a_1,\cdots, a_n)$ such that
$$
  \xi -\epsilon_0 < \frac{1}{n}\sum_{j=1}^{n} \log a_j (x) < \xi+
\epsilon_0.
$$
Then
$$
  E_{\xi}^n(\epsilon_0) = \bigcup_{J \in \mathcal{I}(n,\xi,\epsilon_0)} J.
$$

Hence $\{J:J \in \mathcal{I}(n,\xi,\epsilon_0), n\geq 1 \}$ is a
cover of $E_{\xi}$. When $q(\xi)>0$, by (\ref{estimate}), we have
\begin{eqnarray*}
 & &\sum_{n=1}^{\infty} \sum_{J\in \mathcal{I}(n, \xi, \epsilon_0)}
 |J|^t\\
 &\leq&  \sum_{n=1}^{\infty}  \sum_{(a_1 \cdots a_n)> e^{n(\xi - \epsilon_0)}}
 \frac{e^{nP(t,q(\xi))}}{(a_1\cdots a_n)^{q(\xi)}}\cdot
 \frac{|J|^{t} (a_1 \cdots a_n)^{q(\xi)}}{e^{nP(t,q(\xi))}}
 \\
 &\leq& C \cdot \sum_{n=1}^{\infty}  e^{n(P(t,q(\xi))-(\xi-\epsilon_0)q(\xi))}
 \cdot \sum_{J\in \mathcal{I}(n, \xi, \epsilon_0)} \mu_{t,q(\xi)}(J) < \infty
\end{eqnarray*}
where $C$ is a constant. When $q(\xi)<0$,
\begin{eqnarray*}
 & &\sum_{n=1}^{\infty} \sum_{J\in \mathcal{I}(n, \xi, \epsilon_0)} |J|^t \\
 &\leq & \sum_{n=1}^{\infty}  \sum_{(a_1 \cdots a_n)< e^{n(\xi + \epsilon_0)}}
 \frac{e^{nP(t,q(\xi))}}{(a_1\cdots a_n)^{q(\xi)}}\cdot
 \frac{|J|^{t} (a_1 \cdots a_n)^{q(\xi)}}{e^{nP(t,q(\xi))}}
 \\
 &\leq& C \cdot \sum_{n=1}^{\infty}  e^{n(P(t,q(\xi))-(\xi+\epsilon_0)q(\xi))}
 \cdot \sum_{J\in \mathcal{I}(n, \xi, \epsilon_0)} \mu_{t,q(\xi)}(J) < \infty.
\end{eqnarray*}
Hence we get (\ref{dim-less-t}).

 For the special case $\xi=0$, we need only to show $\dim (E_0)=
 0$. This can be induced by the same process. For any $t>0$, since $\lim_{\xi\to 0} t(\xi) =0$,
 there exists $\xi>0$ such that $0< t(\xi) < t$. We can also choose
 $\epsilon_0>0$ such that
 \begin{eqnarray*}
\frac{P(t,q(\xi))-P(t(\xi),q(\xi))}{q(\xi)} > \epsilon_0.
 \end{eqnarray*}
For $n\geq 1$, set
$$
  E_{0}^n(\epsilon_0):= \bigg\{x \in [0,1)\setminus \mathbb{Q} :  \frac{1}{n}
\sum_{j=1}^{n} \log a_j( x) < \xi+ \epsilon_0 \bigg\}.
$$
We have
$$
    E_{0} \subset \bigcup_{N=1}^{\infty} \bigcap_{n=N}^{\infty }
E_{0}^n(\epsilon_0).
$$
By the same calculation, we get $\dim (E_0) \leq t$. Since $t$ can
be arbitrary small, we obtain $\dim (E_0) =0$.

 By the discussion in the preceding subsection, we have proved
 Theorem \ref{thm-Birkhoff} (1) and (2).

\subsection{Proof of Theorem 1.2 (3) and (4)}

 We are going to investigate  more properties of the functions
$q(\xi)$ and $t(\xi)$.

\begin{pro}\label{q-xi-lim}
 We have
\begin{eqnarray*}
  \lim_{\xi \to 0} q(\xi)= -\infty, \quad \lim_{\xi \to \infty} q(\xi)= 0.
\end{eqnarray*}
\end{pro}
\begin{proof}  We prove the first limit by contradiction. Suppose there
exists a subsequence $\xi_{\delta}\to 0 $ such that
$q(\xi_{\delta})\to M
> -\infty $. Then by (\ref{t-xi-0}) and Proposition \ref{potential}
(3), we have
\begin{eqnarray*}
   \lim_{\xi_{\delta} \to 0} \frac{\partial P}{\partial q}(t(\xi_{\delta}),q(\xi_{\delta}))=
   \frac{\partial P}{\partial q}(0,M) > 0.
\end{eqnarray*}
This contradicts with
\begin{eqnarray*}
   \frac{\partial P}{\partial q}(t(\xi_{\delta}),q(\xi_{\delta}))= \xi_{\delta} \to 0.
\end{eqnarray*}
On the other hand, we know that $q(\xi)\geq 0 $ when $\xi\geq \xi_0
$, and $0 \leq q(\xi) < 2t(\xi) -1 $. Then by (\ref{t-xi-infty}), we
have $\lim\limits_{\xi \to \infty} q(\xi) = 0$. \end{proof}

\medskip

Apply this proposition and (\ref{t'}), combining
(\ref{part-q-infty}) and (\ref{P/q-infty}). We get
\begin{eqnarray*}
   \lim_{\xi \to 0} t'(\xi) 
   =+\infty, \qquad \lim_{\xi \to \infty} t'(\xi) =0.
\end{eqnarray*}
This is the assertion (3) of Theorem \ref{thm-Birkhoff}.
\medskip

Now we will prove the last assertion of Theorem \ref{thm-Birkhoff},
{\it{i.e.}}, $t''(\xi_0)<0$ and there exists $\xi_1>\xi_0 $ such
that $t''(\xi_1)>0 $, basing on the following proposition.

\begin{pro}
For $\xi \in (0, +\infty) $, we have
\begin{eqnarray}\label{q'}
   q'(\xi)= \frac{1- t'(\xi)\frac{\partial ^2 P}{\partial t \partial q}\big( t(\xi),q(\xi)\big)}
   {\frac{\partial ^2 P}{\partial  q^2}\big( t(\xi),q(\xi)\big)};
\end{eqnarray}
\begin{eqnarray}\label{t''}
  t''(\xi)=   \frac{ t'(\xi)^2 \frac{\partial ^2 P}{\partial  t^2}\big( t(\xi),q(\xi)\big)
                       -q'(\xi)^2  \frac{\partial ^2 P}{\partial  q^2}\big( t(\xi),q(\xi)\big)}
                      {-\frac{\partial P}{\partial t}\big( t(\xi),q(\xi)\big)}.
\end{eqnarray}
\end{pro}
\begin{proof} Taking derivative of
 (\ref{E1_1})  with respect to $\xi$, we get
\begin{equation}\label{E1_2}
 t'(\xi)^2 \frac{\partial ^2 P}{\partial  t^2}\big( t(\xi),q(\xi)\big)
 + q'(\xi) t'(\xi)\frac{\partial ^2 P}{\partial q \partial t}\big( t(\xi),q(\xi)\big)
 + t''(\xi)\frac{\partial P}{\partial t}\big( t(\xi),q(\xi)\big) = q'(\xi).
\end{equation}
Taking derivative of the second equation of (\ref{equations_xi})
with respect to $\xi$, we get
\begin{eqnarray}\label{E2_1}
  t'(\xi)\frac{\partial ^2 P}{\partial t \partial q}\big( t(\xi),q(\xi)\big)  +
 q'(\xi) \frac{\partial ^2 P}{\partial  q^2}\big( t(\xi),q(\xi)\big)   = 1,
\end{eqnarray}
which gives immediately (\ref{q'}).

Subtract (\ref{E2_1}) multiplied by $q'(\xi) $ from (\ref{E1_2}), we
get (\ref{t''}). \hfill{$\Box$}
\medskip

We divide the proof of the assertion (4) of  Theorem
\ref{thm-Birkhoff}  into two parts. \vskip 5pt
 \noindent $Proof \  of  \ t''(\xi_0)<0$. By Proposition \ref{potential},
  $\frac{\partial P}{\partial t}(1,0) < 0$. Since $q(\xi_0)=0$,
by (\ref{t'}) we have $t'(\xi_0)=0$. Also by Proposition
\ref{potential}, we get
\begin{eqnarray*}
0< \frac{\partial ^2 P}{\partial  t^2}( t(\xi_0),q(\xi_0))=
\frac{\partial ^2 P}{\partial  t^2}(1,0)< + \infty,
\end{eqnarray*}
and
\begin{eqnarray*}
 0 \leq \frac{\partial ^2 P}{\partial  q^2}(t(\xi_0),q(\xi_0))
 =\frac{\partial ^2 P}{\partial  q^2}(1,0)< + \infty.
\end{eqnarray*}
By (\ref{q'}) and (\ref{t''}), we have
\begin{equation}\label{t''t''}
  t''(\xi)= \frac{ t'(\xi)^2 \frac{\partial ^2 P}{\partial  t^2}( t(\xi),q(\xi))
  \frac{\partial ^2 P}{\partial  q^2}( t(\xi),q(\xi))
  -\left( 1- t'(\xi)\frac{\partial ^2 P}{\partial t \partial q}( t(\xi),q(\xi))\right)^2}
  {-\frac{\partial P}{\partial t}(t(\xi),q(\xi))\frac{\partial ^2 P}{\partial  q^2}( t(\xi),q(\xi))}.
\end{equation}
Thus by $t'(\xi_0)=0 $, we have $t''(\xi_0) < 0$. \hfill{$\Box$}
\medskip

\noindent$Proof \ of  \ t''(\xi_1)>0$. Proposition \ref{q-xi-lim}
shows $\lim\limits_{\xi \to \infty} q(\xi) = 0$ and we know that
$q(\xi_0)= 0 $. However, $q(\xi) $ is not always equal to $0$, so
there exists a $\xi_1 \in [\xi_0, +\infty)$,
 such that $q'(\xi_1) <0$.
Write
\begin{eqnarray*}
H(t,q):= \left(
  \begin{array}{cccc}
    \frac{\partial^2 P}{\partial t^2} & \frac{\partial^2 P}{\partial t \partial q} \\
    \quad &\quad \\
    \frac{\partial^2 P}{\partial t \partial q} & \frac{\partial^2 P}{\partial q^2} \\
  \end{array}
\right)
\end{eqnarray*}
 and add (\ref{E2_1}) multiplied by $q'(\xi) $ to (\ref{E1_2}), we get
 \begin{eqnarray}
 \Big(t'(\xi), q'(\xi)\Big)H(t,q)\big(t'(\xi), q'(\xi)\big)^T
 + \frac{\partial P}{\partial t}\big(t(\xi),q(\xi)\big)t''(\xi)=2q'(\xi).
\end{eqnarray}
Since $H(t,q)$ is definite positive,  $\frac{\partial P}{\partial
t}(t,q) <0$ and $q'(\xi_1) <0$, we have $t''(\xi_1)> 0$. This
completes the proof. \end{proof}

\section{Lyapunov spectrum}
In this last section, we follow the same procedure  as in Section 4
and Section 5 to deduce the Lyapunov spectrum of the Gauss map.
Kesseb\"{o}hmer recently pointed out to us that the Lyapunov
spectrum was also studied by M. Kesseb\"{o}hmer and B. Stratmann
\cite{Kesse2007}.

 Take
$$
F=\Psi = \{\log |\psi'_i|: i \in \mathbb{N}\}.
$$
instead of $ F=\{-\log i: i \in \mathbb{N}\} $ and $\Psi = \{\log
|\psi'_i|: i \in \mathbb{N}\}$. Then the strong H\"older family
becomes $(\tilde{t}-q)\Psi$ and $D $ should be changed to
$$\tilde{D}:= \{(\tilde{t},q): \tilde{t}-q
>1/2 \}.$$ Here and in the rest of this section we will use $\tilde{t}$ instead of $t$ to
distinguish the present situation from that of Khintchine exponents.
What we have done in Section 4 is still useful.
 Denote by $P_1(\tilde{t},q)$ the pressure $P((\tilde{t}-q)\Psi)$. Then
$$P_1(\tilde{t},q)=P(\tilde{t}-q), \quad \mbox{\rm  with} \
P(\cdot)=P(\cdot,0),$$ where $P(\cdot, \cdot)$ is the pressure
function studied in Section 4. Hence $P_1(\tilde{t},q) $ is analytic
and similar equations (\ref{partial_t}) and (\ref{partial_q}) are
obtained just with $\log |T'(x)| $ instead of $\log a_1(x) $.

To determine the Lyapunov spectrum, we begin with the following
proposition which take the place of Proposition \ref{p4.11}.

\begin{pro} For $(\tilde{t},q)\in \tilde{D}$, we have
\begin{eqnarray}\label{inequ-Lya}
  -(\tilde{t}-q)\log 4 +\log \zeta(2\tilde{t}-2q) \leq P_1(\tilde{t},q) \leq \log \zeta(2\tilde{t}-2q).
\end{eqnarray}
Consequently,\\
\indent {\rm (1)} for any point $(\tilde{t}_0,q_0)$ on the line
$\tilde{t}-q=1/2$,
$$
\lim_{(\tilde{t},q)\to (\tilde{t}_0,q_0)}P(\tilde{t},q)=\infty;
$$
 \indent {\rm (2)} for fixed $\tilde{t}\in \mathbb{R}$,
\begin{eqnarray*}
  \lim_{q\to \tilde{t}-\frac{1}{2}} \frac{\partial P_1}{\partial q}(\tilde{t},q) =
  +\infty;
\end{eqnarray*}
 \indent {\rm (3)} recalling $\gamma_0= 2\log
\frac{1+\sqrt{5}}{2} $,  for fixed $\tilde{t}\in \mathbb{R}$,
\begin{eqnarray*}
  \lim_{q\to -\infty} \frac{P_1(\tilde{t},q)}{q} = \gamma_0, \qquad
  \lim_{q\to -\infty} \frac{\partial P_1}{\partial q}(\tilde{t},q) = \gamma_0.
\end{eqnarray*}

\end{pro}
\begin{proof} $P_1(\tilde{t},q)$ is defined as
\begin{eqnarray*}
P_1(\tilde{t},q):=\lim_{n\to \infty} \frac{1}{n} \log
\sum_{\omega_1=1}^{\infty}\cdots\sum_{\omega_n=1}^{\infty} \exp
\left(
  \sup_{x\in [0,1]} \log \prod_{j=1}^{n}
([\omega_j,\cdots, \omega_n+x])^{2(\tilde{t}-q)}\right).
\end{eqnarray*}
The proofs of (1) and (2) are the same as in the proof of
Proposition \ref{p4.11}.

To get (3), we follow another method. Since
$P_1(\tilde{t},q)=P(\tilde{t}-q)$, we need only to show
$$\lim_{\tilde{t}\to \infty} P'(\tilde{t})=-\gamma_0,  \  \
P(\tilde{t}) + \tilde{t}\gamma_0 = o (\tilde{t}) \quad (\tilde{t}\to
\infty).$$

By Proposition \ref{p4.9}, $P(\tilde{t})$ is analytic on
$(1/2,\infty)$. Let $E:= \{P'(\tilde{t}) : \tilde{t}>1/2\}$, denote
by $\mbox{\rm Int}(E)$ and $\mbox{\rm Cl}(E) $ the interior and
closure of $E$. By a result in \cite{Jen1}, we have
\begin{eqnarray*}
 \mbox{\rm Int}(E) \subset \left\{ -\int \log |T'(x)| d\mu : \mu \in
 \mathcal{M} \right\} \subset \mbox{\rm Cl}(E),
\end{eqnarray*}
where $\mathcal{M}$ is the set of the invariant measures on $[0,1]$.
By Birkhoff's theorem, for any $\mu \in \mathcal{M}$, we have
\begin{eqnarray*}
\int \lambda(x) d \mu =\int \log |T'(x)| d\mu.   
\end{eqnarray*}
 However, we know
that $\lambda(x)\geq \gamma_0=2\log \frac{1+\sqrt{5}}{2}$. Thus
\begin{eqnarray}\label{less-gamma0}
-\int \log |T'(x)| d\mu  \leq - \gamma_0 \qquad \forall\mu \in
 \mathcal{M}.
\end{eqnarray}
Let $\theta_0=\frac{\sqrt{5}-1}{2}$. Then $T(\theta_0)=\theta_0$ and
the Dirac measure $\mu=\delta_{\theta_0}$  is invariant, and
\begin{eqnarray*}
-\int \log |T'(x)| d\delta_{\theta_0} = - \log |T'(\theta_0)| =
-\gamma_0.
\end{eqnarray*}
However, by the continuity of $P'$, we know that $E$ is an interval.
Therefore $-\gamma_0$ is the right endpoint of $E$. Since
$P'(\tilde{t})$ is increasing, we get $$\lim_{\tilde{t}\to \infty}
P'(\tilde{t}) = - \gamma_0.$$

Let $\{\beta_n\}_{n\geq 1}$ be such that $\beta_n < -\gamma_0$ and
$\lim\limits_{n\to \infty} \beta_n=-\gamma_0$. There exist $t_n\in
\mathbb{R}$ such that $t_n\to \infty$ and $P'(t_n)=\beta_n$. By the
variational principle (\cite{Wa3}, see also \cite{Ma3}), there
exists an ergodic measure $\mu_{t_n}$ such that
\begin{eqnarray*}
P(t_n)=h_{\mu_{t_n}} - t_n \int \log |T'|(x) d\mu_{t_n},
\end{eqnarray*}
where $h_{\mu_{t_n}}$ stands for the metric entropy of $\mu_{t_n}$.
By the compactness of $\mathcal{M}$ there exists an invariant
measure $\mu_{\infty}$ which is the weak limit of $\mu_{t_n}$(more
precisely some subsequence of $\mu_{t_n}$. But, without loss of
generality, we write it as $\mu_{t_n}$). By the semi-continuity of
metric entropy,  for any $\epsilon
>0$ we have $h_{\mu_{t_n}}\leq h_{\mu_{\infty}}+\epsilon$ when $t_n$ is
large enough. Thus by (\ref{less-gamma0}),
\begin{eqnarray*}
P(t_n)\leq h_{\mu_{\infty}}+\epsilon - t_n\gamma_0.
\end{eqnarray*}
We will show that $h_{\mu_{\infty}}=0$ (see the next lemma), which
will imply
\begin{eqnarray*}
P(t_n)\leq \epsilon - t_n\gamma_0.
\end{eqnarray*}
However, by the definition of $P_1(\tilde{t},q)$, $P(\tilde{t})$ can
be written as
\begin{eqnarray*} P(\tilde{t})=\lim_{n\to \infty}
\frac{1}{n} \log
\sum_{\omega_1=1}^{\infty}\cdots\sum_{\omega_n=1}^{\infty} \exp
\Bigg(
  \sup_{x\in [0,1]} \log \prod_{j=1}^{n}
([\omega_j,\cdots, \omega_n+x])^{2\tilde{t}}\Bigg).
\end{eqnarray*}
Thus if we just take one term in the summation, we have
\begin{eqnarray*} P(\tilde{t})\geq \lim_{n\to \infty}
\frac{1}{n} \log \exp \Bigg(
  \sup_{x\in [0,1]} \log \prod_{j=1}^{n}
([\underbrace{1,\cdots,1}_{n-j},
1+x])^{2\tilde{t}}\Bigg)=-\tilde{t}\gamma_0.
\end{eqnarray*}
Hence we get
\begin{eqnarray*}
P(\tilde{t}) + \tilde{t}\gamma_0 = o(\tilde{t}) \qquad (\tilde{t}\to
\infty).
\end{eqnarray*}
\end{proof}

 Now we are led to show
\begin{lem}
$h_{\mu_{\infty}}=0.$
\end{lem}
\begin{proof}
 Let $h_{\mu_{\infty}}(x)$ be the local entropy
  of $\mu_{\infty}$ at $x$ which is defined by
\begin{eqnarray*}
  h_{\mu_{\infty}}(x)= \lim_{n\to \infty} \frac{\log \mu_{\infty}(I_n(x))}{n},
\end{eqnarray*}
if the limit exists. Let $\underline{D}_{\mu_{\infty}}(x)$ be the
lower local dimension of $\mu_{\infty}$ at $x$ which is defined by
\[
\underline{D}_{\mu_{\infty}}(x):= \liminf_{r\to 0} \frac{\log
\mu_{\infty}(B(x,r))}{\log r}.
\]
By Shannon-McMillan-Breiman theorem, $h_{\mu_{\infty}}(x)$ exists
$\mu_{\infty}$-almost everywhere. It is also known that $\lambda(x)$
exists almost everywhere (by Birkhoff's theorem). So, by the
definitions, we have
\[
      h_{\mu_{\infty}}(x) = \underline{D}_{\mu_{\infty}}(x)  \lambda(x)   \quad
      \mu_{\infty}-a.e..
\]

By Birkhoff's theorem and (\ref{P'(t)}),
\begin{eqnarray*}
 \int \lambda(x) d{\mu_{\infty}}(x)&=&\int \log |T'|(x) d{\mu_{\infty}}(x) \\
 &=& \lim_{n\to \infty} \int \log |T'|(x) d\mu_{t_n}\\
 &=& -\lim_{n\to \infty} P'(t_n) =
\gamma_0 <\infty.
\end{eqnarray*}
Hence $\lambda(x)$ is almost everywhere finite. Recall that
\cite{BK}
\[
      h_{\mu_{\infty}} = \int h_{\mu_{\infty}}(x) d {\mu_{\infty}} (x).
\]
 Thus it suffices to prove
\[
       \underline{D}_{\mu_{\infty}}(x) = 0   \quad
       {\mu_{\infty}}-a.e..
\]
 That means (\cite{Fan}) the upper dimension of ${\mu_{\infty}}$ is zero,
i.e., ${\mu_{\infty}}$ is supported by a zero-dimensional set.

Since $\int \lambda(x) d{\mu_{\infty}}(x)=\gamma_0 $ and
$\lambda(x)\geq \gamma_0 $ for any $x$, we have for $\mu_{\infty} $
almost everywhere $\lambda(x)=\gamma_0 $. Thus by Birkhoff's
theorem, $\mu_{\infty}$ is supported by the following set
\begin{eqnarray}\label{set}
\bigg\{ x\in [0,1]: \lim_{n\to \infty} \frac{1}{n}\sum_{j=0}^{n-1}
\log |T'(T^j x)|=\gamma_0 \bigg\}.
\end{eqnarray}
Thus we need only to show that the Hausdorff dimension of this set
is zero.

Recall
\begin{eqnarray*}
 \lim_{n\to \infty} \frac{1}{n}\sum_{j=0}^{n-1} \log |T'(T^j x)| =
 2\lim_{n\to \infty}\frac{1}{n} \log q_n(x).
\end{eqnarray*}
By Lemma \ref{ABC}, (\ref{set}) is in fact the following
\begin{eqnarray}\label{set2}
\bigg\{ x\in [0,1]: \lim_{n\to \infty} \frac{1}{n}\sum_{j=1}^{n}
\log a_j(x) =0 \bigg\}.
\end{eqnarray}
However, the Hausdorff dimension of (\ref{set2}) is nothing but
$t(0) $, the special case $\xi=0 $ discussed in the subsection 5.1.,
which was proved to be zero. Thus the proof is completed.
\end{proof}

\medskip

Recall that $\lambda_0= \int \log |T'(x)| d \mu_G $. Let
$\tilde{D}_0:= \{(\tilde{t},q): \tilde{t}-q>1/2, 0\leq \tilde{t}
\leq 1 \} $. We have a  proposition similar to Proposition
\ref{solution}.
\begin{pro}\label{solution-Lya}
 For any $\beta\in (\gamma_0, \infty)$, the system
\begin{eqnarray}\label{equations-Lya}
\left\{
  \begin{array}{ll}
     P_1(\tilde{t},q)=q\beta, \\
     \displaystyle \frac{\partial P_1}{\partial q}(\tilde{t},q)= \beta
  \end{array}
\right.
\end{eqnarray}
admits a unique solution $(\tilde{t}(\beta),q(\beta))\in
\tilde{D}_0$. For $\beta=\lambda_0 $, the solution is
$(\tilde{t}(\lambda_0),q(\lambda_0))=(1,0) $. The functions
$\tilde{t}(\beta)$ and $q(\beta)$ are analytic.
\end{pro}

With the same argument, we can prove that $\tilde{t}(\beta) $ is the
spectrum of Lyapunov exponent. It is analytic, increasing on
$(\gamma_0, \lambda_0] $ and decreasing on $(\lambda_0, \infty) $.
It is also neither concave nor convex. In other words, Theorem
\ref{thm-Lyapunov} can be similarly proved.
\begin{center}
\begin{pspicture}(0,0)(10,7)
\rput(4.97,3){\includegraphics{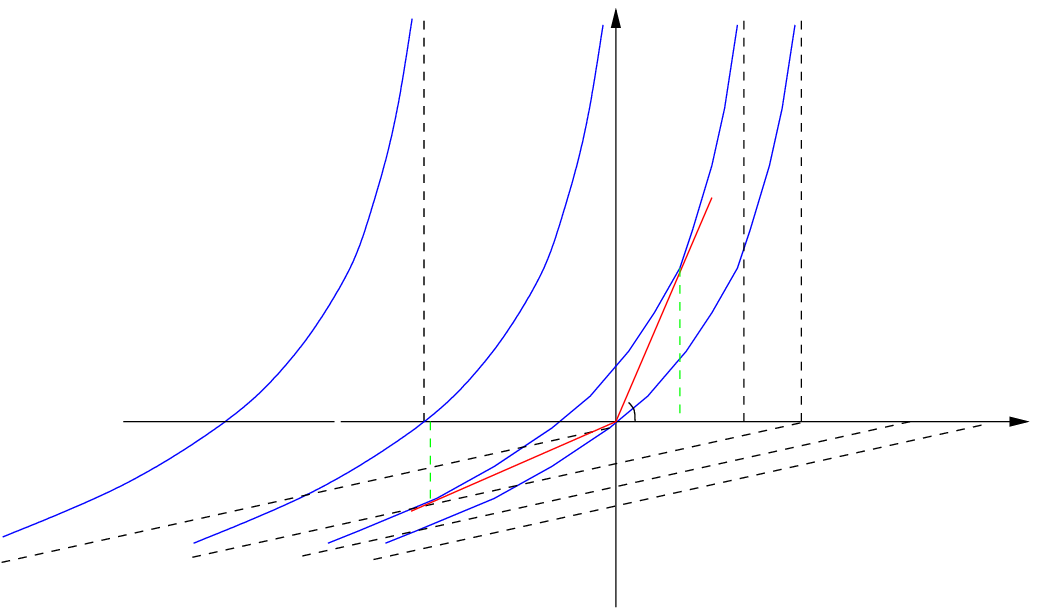}}
\rput(8.04,2.1){$\frac{1}{2}$} \rput(3.65,2.1){$-\frac{1}{2}$}
\rput(6.3,2){$\beta$} \rput(3.25,5.2){$t=0$}
\rput(5.0,4.2){$t=\frac{1}{2}$} \rput(6.6,4.4){$t(\beta)$}
\rput(8.3,5.75){$t=1$}
\rput(6.15,1.6){$0$} \rput(10,1.6){$q$} \rput(6,6.3){$P(t,q)$}
\end{pspicture}
\begin{center}
{{\it Figure 4.} Solution of (\ref{equations-Lya})}
\end{center}
\end{center}

We finish the paper by the observation that the Lyapunov spectrum
 can be stated as follows, which is similar to the classic formula, but with the difference that we have to divide the Legendre transform
 by $\beta$.
\begin{pro}
\begin{eqnarray}\label{t-beta}
   \tilde{t}(\beta)= \frac{P(-q(\beta))}{\beta}-q(\beta) =
   \frac{1}{\beta} \inf_{q} \{ P(-q) - q\beta\}.
\end{eqnarray}
\end{pro}
\begin{proof}  In fact, the family of functions $P_1(\tilde{t},q) $
with parameter $\tilde{t}$ are just right translation of the
function $P(-q)$ with the length $\tilde{t}$. Write the system
(\ref{equations-Lya}) as follows
\begin{eqnarray}\label{equations-Lya1}
\left\{
  \begin{array}{ll}
    P(\tilde{t}-q)=q\beta , \\
     \frac{d P}{d q}(\tilde{t}-q)= \beta.
  \end{array}
\right.
\end{eqnarray}
If we denote by $\mu_q $, the Gibbs measure with respect to
potential $q\Psi $, then by a left translation the system
(\ref{equations-Lya1}) can be written as
\begin{eqnarray*}
\left\{
  \begin{array}{ll}
    P(-q)=(\tilde{t}+q) \beta  , \\
     \frac{d P}{d q}(-q)= \beta.
  \end{array}
\right.
\end{eqnarray*}
Thus
\begin{eqnarray*}
\left\{
  \begin{array}{ll}
    \tilde{t}=\frac{P(-q)}{\beta} -q , \\
     \frac{d P}{d q}(-q)= \beta.
  \end{array}
\right.
\end{eqnarray*}
By using the second equation, we can write $q$ as a function of
$\beta$, hence we get (\ref{t-beta}).\end{proof}

\begin{center}
\begin{pspicture}(0,0)(10,7)
\rput(3.08,2.98){\includegraphics{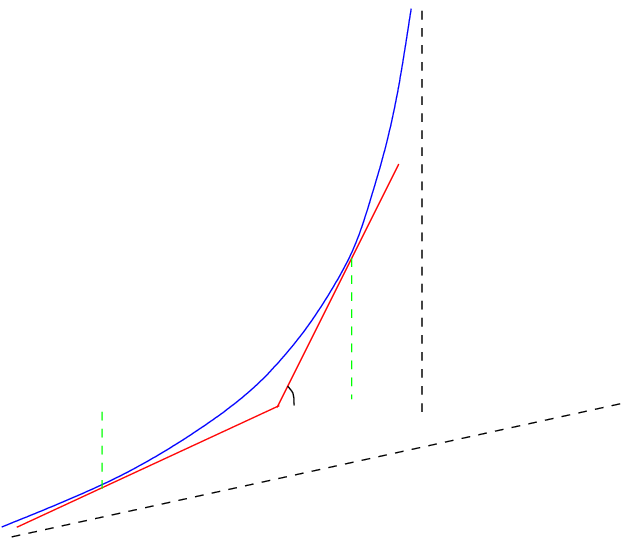}}
\rput(4.71,3.16){\includegraphics{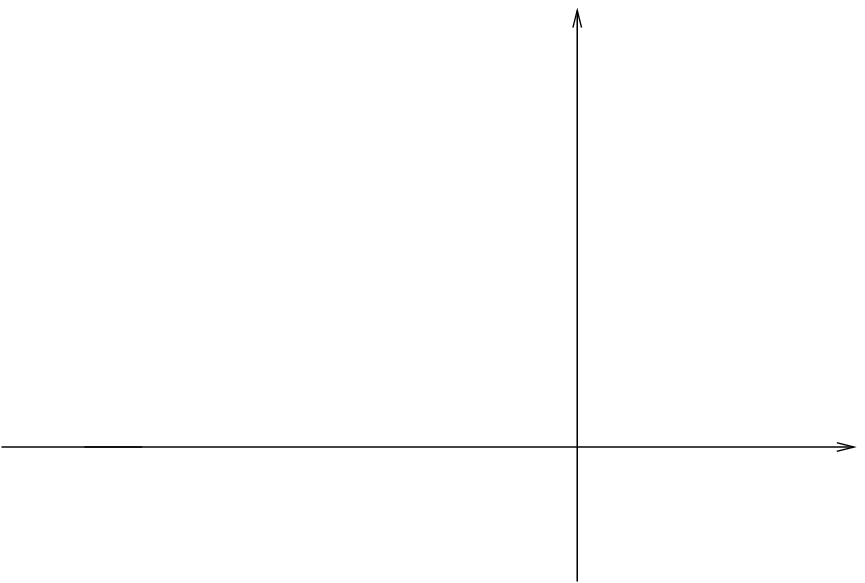}}
\rput(6.56,2.98){\includegraphics{Legendre.eps}}
\rput(4.46,1.9){$-\frac{1}{2}$} \rput(6.5,1.85){$\beta$}
\rput(3.05,1.85){$\beta$} \rput(3.43,5){$t=0$} \rput(4,6){$P(-q)$}
\rput(7,5){$t(\beta)$} \rput(8,6){$P(t-q)$} \rput(6.35,1.4){$0$}
\rput(2.05,1.75){$-1$} \rput(8.95,1.85){$q$}
\rput(6.15,6.3){$P_1(t,q)$}
\end{pspicture}
\begin{center}
{{\it Figure 5.} The other way to see $t(\beta) $}
\end{center}
\end{center}

\end{document}